\documentclass[review,onefignum,onetabnum]{siamart190516}

\usepackage{amsmath,amssymb}
\usepackage{multirow}
\usepackage{lipsum}
\usepackage{amsfonts}
\usepackage{graphicx}
\usepackage{epstopdf}
\usepackage{algorithmic}
\ifpdf
  \DeclareGraphicsExtensions{.eps,.pdf,.png,.jpg}
\else
  \DeclareGraphicsExtensions{.eps}
\fi

\newsiamremark{remark}{Remark}
\newsiamremark{hypothesis}{Hypothesis}
\crefname{hypothesis}{Hypothesis}{Hypotheses}
\newsiamthm{claim}{Claim}

\headers{Uniform error bounds of EWI for NKGE}{Yue Feng and Wenfan Yi}

\title{Uniform error bounds of an exponential wave integrator for the long-time dynamics of the nonlinear Klein-Gordon equation\thanks{Submitted to the editors DATE.}}

\author{Yue Feng\thanks{Department of Mathematics, National University of Singapore, Singapore 119076
  (\email{fengyue@u.nus.edu}).}
\and Wenfan Yi\thanks{School of Mathematics, Hunan University, Changsha, 410082, Hunan Province, P. R. China
  (\email{wfyi@hnu.edu.cn}).}
}

\usepackage{amsopn}

\ifpdf
\hypersetup{
  pdftitle={Uniform error bounds of an exponential wave integrator for the long-time dynamics of the nonlinear Klein-Gordon equation},
  pdfauthor={Y. Feng and W. Yi}
}
\fi

\graphicspath{{figures/}}

\begin{document}
\nolinenumbers

\maketitle

% REQUIRED
\begin{abstract}
We establish uniform error bounds of an exponential wave integrator Fourier pseudospectral (EWI-FP) method for the long-time dynamics of the nonlinear Klein-Gordon equation (NKGE) with a cubic nonlinearity whose strength is characterized by $\varepsilon^2$ with $\varepsilon \in (0, 1]$ a dimensionless parameter. When $0 < \varepsilon \ll 1$, the problem is equivalent to the long-time dynamics of the NKGE with small initial data (and $O(1)$ cubic nonlinearity), while the amplitude of the initial data (and the solution) is at $O(\varepsilon)$. For the long-time dynamics of the NKGE up to the time at $O(1/\varepsilon^{2})$, the resolution and error bounds of the classical numerical methods depend significantly on the small parameter $\varepsilon$, which causes severe numerical burdens as $\varepsilon \to 0^+$. The EWI-FP method is fully explicit, symmetric in time and has many superior properties in solving wave equations. By adapting the energy method combined with the method of mathematical induction, we rigorously carry out the uniform error bounds of the EWI-FP discretization at  $O(h^{m_0} + \varepsilon^{2-\beta}\tau^2)$ up to the time at $O(1/\varepsilon^{\beta})$ with $0 \leq \beta \leq 2$, mesh size $h$, time step $\tau$ and $m_0$ an integer depending on the regularity of the solution. By a rescaling in time, our results are straightforwardly extended to the error bounds and $\varepsilon$-scalability (or meshing strategy requirement) of the EWI-FP method for an oscillatory NKGE, whose solution propagates waves with wavelength at $O(1)$ and $O(\varepsilon^{\beta})$ in space and time, respectively, and wave speed at $O(\varepsilon^{-\beta})$. Finally, extensive numerical results are reported to confirm our error estimates.
\end{abstract}

% REQUIRED
\begin{keywords}
nonlinear Klein-Gordon equation, Gautschi-type exponential wave integrator Fourier pseudospectral method, uniform error bounds, weak nonlinearity, long-time analysis
\end{keywords}

% REQUIRED
\begin{AMS}
35L70, 65M12, 65M15, 65M70, 81-08
\end{AMS}

%%%%%%%%%%%%%%%%%%%
% %    Section 1  Introduction
%%%%%%%%%%%%%%%%%%%
\section{Introduction}
The nonlinear Klein-Gordon equation (NKGE) is well-known as the relativistic (and nonlinear) version of the Schr{\"o}dinger equation and plays a significant role in many scientific fields such as quantum electrodynamics, nonlinear optics, particle, plasma  and/or  solid state physics to describe the dynamics of the spinless particle like the pion \cite{CHL, DEGM, SJJ, SB, S, XG}. We consider the following NKGE with a cubic nonlinearity on a torus $\mathbb{T}^d (d = 1, 2, 3)$ as
\begin{equation}
\left\{
\begin{aligned}
&\partial_{tt} u({\bf{x}}, t) - \Delta u({\bf{x}}, t) + u({\bf{x}}, t) + \varepsilon^{2} u^3({\bf{x}}, t) = 0,\quad {\bf{x}} \in \mathbb{T}^d,\quad t > 0,\\
&u({\bf{x}}, 0) = \phi({\bf{x}}), \quad \partial_t u({\bf{x}}, 0) = \gamma({\bf{x}}) , \quad {\bf{x}} \in \mathbb{T}^d.
\end{aligned}\right.
\label{eq:WNE}
\end{equation}
Here, {\bf{x}} is the spatial coordinate, $t$ is time, $u := u({\bf{x}},t)$ is a real-valued scalar field, $\varepsilon \in (0, 1]$ is a dimensionless parameter and the initial data $\phi({\bf{x}})$ and $\gamma({\bf{x}})$ are two given real-valued functions independent of the parameter $\varepsilon$. Provided that $u(\cdot, t) \in H^1(\mathbb{T}^d)$ and $\partial_t u(\cdot, t) \in L^2(\mathbb{T}^d)$, the NKGE \eqref{eq:WNE} is time symmetric or time reversible and conserves the energy \cite{CHL, HLu} as
\begin{equation*}
\begin{split}
E(t) &:= E(u(\cdot, t)) =\int_{\mathbb{T}^d} \Big[ |\partial_t u ({\bf{x}}, t)|^2 + |\nabla u({\bf{x}}, t)|^2 + |u({\bf{x}}, t)|^2 + \frac{\varepsilon^2}{2} |u({\bf{x}}, t)|^4 \Big] d {\bf{x}} \\
& \equiv  \int_{\mathbb{T}^d} \Big[ |\gamma({\bf{x}})|^2 + |\nabla \phi({\bf{x}})|^2 + |\phi({\bf{x}})|^2 + \frac{\varepsilon^2}{2} |\phi({\bf{x}})|^4  \Big] d {\bf{x}} \\
&= E(0) = O(1), \qquad t \geq 0.
\end{split}
\end{equation*}

When $0 < \varepsilon \ll1 $, rescaling the amplitude of the wave function $u({\bf{x}}, t)$ by introducing $w({\bf{x}},t)  = \varepsilon u({\bf{x}},t)$, the NKGE \eqref{eq:WNE} with weak nonlinearity (and initial data with amplitude at $O(1)$) can be reformulated as the following NKGE with small initial data (and $O(1)$ cubic nonlinearity):
\begin{equation}
\left\{
\begin{aligned}
&\partial_{tt} w({\bf{x}},t)  -\Delta w({\bf{x}},t)  +   w({\bf{x}},t)  + w^3({\bf{x}},t) = 0, \quad {\bf{x}} \in \mathbb{T}^d,\quad t > 0,\\
&w({\bf{x}}, 0) = \varepsilon \phi( {\bf{x}}),\quad \partial_t w({\bf{x}}, 0) = \varepsilon \gamma({\bf{x}}),\quad {\bf{x}} \in \mathbb{T}^d.
\label{eq:SIE}
\end{aligned}\right.
\end{equation}
Noticing that the amplitude of the initial data (and the solution) of the NKGE \eqref{eq:SIE} is at $O(\varepsilon)$. Again, the above NKGE \eqref{eq:SIE} is time symmetric or time reversible and conserves the energy as
\begin{equation*}
\begin{split}
\widetilde{E}(t) &:= \widetilde{E}(w(\cdot, t)) =  \int_{\mathbb{T}^d} \Big[ |\partial_t w ({\bf{x}}, t)|^2 + |\nabla w({\bf{x}}, t)|^2 + |w({\bf{x}}, t)|^2 + \frac{1}{2} |w({\bf{x}}, t)|^4 \Big] d {\bf{x}} \\
& \equiv  \int_{\mathbb{T}^d} \Big[ \varepsilon^2|\gamma({\bf{x}})|^2 + \varepsilon^2|\nabla \phi({\bf{x}})|^2 + \varepsilon^2|\phi({\bf{x}})|^2 + \frac{\varepsilon^4}{2} |\phi({\bf{x}})|^4  \Big] d {\bf{x}} \\
& = \varepsilon^2 E(0) = O(\varepsilon^2), \qquad t \geq 0.
\end{split}
\end{equation*}
We remark here that the long-time dynamics of the NKGE with weak nonlinearity and $O(1)$ initial data, i.e. the NKGE \eqref{eq:WNE}, is equivalent to the long-time dynamics of the NKGE with small initial data and $O(1)$ nonlinearity, i.e. the NKGE (\ref{eq:SIE}). In the following, we only present numerical methods and their error estimates for the NKGE \eqref{eq:WNE} with weak nonlinearity. Extensions of the numerical methods and their error bounds to the NKGE \eqref{eq:SIE} with small initial data are straightforward.

The NKGE \eqref{eq:WNE} (or  \eqref{eq:SIE}) has been studied extensively in both analytical and numerical aspects. Along the analytical front, the existence of global classical solutions and almost periodic solutions have been investigated \cite{BJ, BV, CE, VH} and references therein. For the Cauchy problem with small initial data (or weak nonlinearity), the asymptotic behavior of solutions and long-time conservation properties have been considered in the literatures \cite{CHL, CHL1, KT, K, O, TY}. In recent years, the life-span of the solution to the NKGE (\ref{eq:SIE}) (or \eqref{eq:WNE}) has gained a surge of attention \cite{KT, LH}. The analytical results indicate that the life-span of a smooth solution to the NKGE (\ref{eq:SIE}) (or \eqref{eq:WNE}) is at least up to the time at $O(\varepsilon^{-2})$ \cite{D}. For more details related to this topic, we refer to \cite{D2, FZ} and references therein. In the numerical aspects, various numerical schemes have been developed in the literatures for the NKGE (\ref{eq:WNE}) (or (\ref{eq:SIE})), including the finite difference time domain (FDTD) methods \cite{BD, DB, SV}, exponential wave integrator Fourier pseudospectral (EWI-FP) method \cite{BD, BDZ, BY}, asymptotic-preserving (AP) schemes \cite{CCLM, FS, J}, multiscale time integrator Fourier pseudospectral (MTI-FP) method \cite{BCZ, BDZ0}, etc. However, to the best of our knowledge, the error estimates are normally valid up to the fixed time in the previous works. Since the life-span of the solution to the NKGE (\ref{eq:WNE}) (or (\ref{eq:SIE})) is up to the time at $O(\varepsilon^{-2})$, one has to extend the classical error bounds of the numerical method for the NKGE (\ref{eq:WNE}) (or (\ref{eq:SIE})) up to the time at $O(\varepsilon^{-2})$ instead of  $O(1)$, i.e. long-time error analysis. In our recent work \cite{BFY}, by employing the energy method, mathematical induction or cut-off technique, the long-time error bounds of four explicit/semi-implicit/implicit conservative/nonconservative FDTD methods have been rigorously established for the NKGE \eqref{eq:WNE} with particular attentions paid to how the error bounds depend explicitly on the mesh size $h$ and the time step $\tau$ as well as the small parameter $\varepsilon \in (0, 1]$ up to the time at $O(\varepsilon^{-\beta})$ with $0 \leq \beta \leq 2$. Based on our error bounds, in order to obtain ``correct'' numerical approximations of the NKGE \eqref{eq:WNE} up to the time  at $O(\varepsilon^{-\beta})$, the $\varepsilon-$scalability (or meshing strategy requirement) of the FDTD methods should be:
\begin{equation*}
h = O(\varepsilon^{\beta/2}) \quad \mbox{and} \quad \tau = O(\varepsilon^{\beta/2}),\end{equation*}
which suggests that the FDTD methods are {\textbf{under-resolution}} in both space and time with respect to $\varepsilon \in (0, 1]$ regarding to the Shannon's information theorem \cite{Landau,Shannon1,Shannon2}- to resolve a wave one needs to use a few grid points per wavelength. Moreover, the Crank-Nicolson finite difference method (CNFD) is fully implicit and depends on a direct solver or an iterative solver which is quite time-consuming, while other FDTD methods naturally suffer from stability problems. The fourth-order compact finite difference (4cFD) method has also been used for the long-time dynamics of the NKGE \eqref{eq:WNE}, which has better spatial resolution than the FDTD methods \cite{Feng}. Thus, a discretization that performs well and allows larger mesh size and time step for a given $\varepsilon$, is of great importance for the investigation of the long-time dynamics of the NKGE \eqref{eq:WNE} (or (\ref{eq:SIE})).

The EWI-FP method has many superior properties and is widely used to solve wave equations in different regimes \cite{BC1, BCJY, BD, BDZ, DXZ, YRS}. The key ingredients of the EWI-FP method are: (i) applying the Fourier pseudospectral discretization for the spatial derivatives;  (ii) adapting an exponential wave integrator based on some efficient quadrature rules for integrating the NKGE under proper transmission conditions between different time intervals in  phase space. The main purpose of this paper is to carry out long-time error bounds of a Gautschi-type EWI-FP discretization for the NKGE \eqref{eq:WNE} (or (\ref{eq:SIE}))  up to the time at $O(\varepsilon^{-2})$. In our numerical analysis, besides the standard technique of the energy method and inverse inequality, we adapt the mathematical induction to obtain a priori uniform bound of the numerical solution. The rigorous error bounds suggest that the $\varepsilon-$scalability of the EWI-FP method up to the time at $O(\varepsilon^{-2})$ should be:
\begin{equation*}
h = O(1) \quad \mbox{and} \quad \tau = O(1),
\end{equation*}
which is uniform in terms of $\varepsilon$ and performs much better than the FDTD and 4cFD methods. In addition, by rescaling the time as $t \to t/\varepsilon^{\beta}$, the NKGE \eqref{eq:WNE} can be reformulated as an oscillatory NKGE whose solution propagates waves with wavelength at $O(1)$ and $O(\varepsilon^{\beta})$ in space and time, respectively, and wave speed at $O(\varepsilon^{-\beta})$. The rapid oscillation in time and outgoing waves bring new challenges in the design and analysis for the numerical methods. Based on the long-time error analysis for the EWI-FP method to the NKGE \eqref{eq:WNE}, the error bounds and $\varepsilon$-scalability of the EWI-FP method for the oscillatory NKGE in a fixed time interval can be obtained straightforwardly. 

The rest of this paper is organized as follows. In Section 2, we discuss the derivation and analyze the stability of the EWI-FP discretization for the NKGE \eqref{eq:WNE}. In Section 3, we establish uniform error bounds of the EWI-FP method for the long-time dynamics of the NKGE \eqref{eq:WNE} up to the time at $O(\varepsilon^{-\beta})$ with $0 \leq \beta \leq 2$. Extensive numerical results are reported in Section 4 to confirm our error bounds. In Section 5, we extend the EWI-FP method and the error estimates to an oscillatory NKGE in a fixed time interval. Finally, some conclusions are drawn in Section 6. Throughout the paper, we adopt the standard Sobolev spaces and use the notation $p \lesssim q$ to represent that there exists a generic constant $C > 0$ which is independent of $h$, $\tau$ and $\varepsilon$ such that $|p| \leq Cq$.

%%%%%%%%%%%%%%%%%%%%%%%%
% %    Section 2  Numerical methods
%%%%%%%%%%%%%%%%%%%%%%%%
\section{An exponential wave integrator spectral method}
In this section, inspired by the ideas in \cite{BC1, BD, BDZ, Gr1, Gr2, HL}, we propose and analyze a Gautschi-type exponential wave integrator Fourier spectral/pseudospectral (EWI-FS/EWI-FP) discretization for the NKGE (\ref{eq:WNE}). For simplicity of notations, we only present the numerical methods in one space dimension (1D). Thanks to tensor grids, generalizations to higher dimensions are straightforward and results remain valid with minor modifications.  In 1D, the NKGE (\ref{eq:WNE}) collapses to
\begin{equation}
\left\{
\begin{aligned}
&\partial_{tt} u(x, t) - \partial_{xx} u(x, t) + u(x, t) + \varepsilon^{2} u^3(x, t) = 0,\quad x \in \Omega = (a, b),\quad t > 0,\\
&u(x, 0) = \phi(x), \quad \partial_t u(x, 0) = \gamma(x), \quad x \in \overline{\Omega} = [a, b],
\label{eq:WNE_1D}
\end{aligned}\right.
\end{equation}
 with periodic boundary conditions.

\subsection{The method}

Choose the time step $\tau = \Delta t > 0$ and mesh size $h := \Delta x = (b-a)/M$ with $M$ being an even positive integer, and denote time steps by $t_n = n\tau$ for $n \geq 0$ and grid points as $x_j := a+j h$, $j = 0, 1, \cdots, M$. Denote the index set $\mathcal{T}_M = \{l | l = -M/2,-M/2+1, \cdots, M/2-1\}$, and define $C_p(\Omega) = \{u \in C(\Omega) | u(a) = u(b)\}$ and
\begin{equation*}
\begin{split}
& X_M := \mbox{span}\{e^{i\mu_l(x-a)}, \mu_l = \frac{2 \pi l}{b-a}, x \in \overline{\Omega}, l \in \mathcal{T}_M \}, \\
&Y_M := \{u=(u_0, u_1, \cdots, u_M) \in \mathbb{R}^{M+1} : u_0 = u_M\}.
\end{split}
\end{equation*}
  For any $u(x) \in C_p(\Omega)$ and a vector $u \in Y_M$, define $P_M : L^2(\Omega) \to X_M$ as the standard $L^2-$projection operator onto $X_M$, $I_M : C_p(\Omega) \to X_M$ or $I_M : Y_M \to X_M$ as the trigonometric interpolation operator \cite{HGG, ST}, i.e.,
\begin{equation*}
(P_M u)(x) = \sum_{l \in \mathcal{T}_M} \widehat{u}_l e^{i\mu_l (x-a)}, \quad (I_M u)(x) = \sum_{l \in \mathcal{T}_M} \widetilde{u}_l e^{i\mu_l (x-a)}, \quad x \in \overline{\Omega},
\end{equation*}
where
\begin{equation}
\widehat{u}_l = \frac{1}{b-a} \int^{b}_{a} u(x) e^{-i\mu_l (x-a)} dx, \quad \widetilde{u}_l = \frac{1}{M} \sum^{M-1}_{j=0} u_j e^{-i\mu_l (x_j-a)},\quad l \in \mathcal{T}_M,
\label{eq:def_hat}
\end{equation}
with $u_j$ interpreted as $u(x_j)$ when involved.

We begin with the Fourier spectral method for discretizing the NKGE (\ref{eq:WNE_1D}) in space, i.e., find
\begin{equation}
u_M(x, t) = \sum_{l \in \mathcal{T}_M} \widehat{(u_M)}_l(t) e^{i\mu_l(x-a)}\in X_M, \quad x \in \overline{\Omega}, \quad t \geq 0,
\label{eq:um}
\end{equation}
such that
\begin{equation}
\partial_{tt} u_M(x, t) - \partial_{xx} u_M(x, t)+ u_M(x, t) + \varepsilon^{2} P_M f(u_M(x, t)) = 0,\quad x \in \overline{\Omega},\quad t \geq 0,\\
\label{eq:21um}
\end{equation}
with $f(v) = v^3$. Plugging (\ref{eq:um}) into (\ref{eq:21um}) and noticing the orthogonality of the Fourier basis functions, we find
\begin{equation}
\frac{d^2}{dt^2} \widehat{(u_M)}_l(t) + (1+\mu^2_l) \widehat{(u_M)}_l(t) + \varepsilon^{2} \widehat{(f(u_M))}_l(t) = 0,\quad  l \in \mathcal{T}_M,\quad t \geq 0.
\end{equation}
For each fixed $l (l \in \mathcal{T}_M)$, when $t$ is near $t = t_n (n \geq 0)$, the above ODEs can be rewritten as
\begin{equation}
\frac{d^2}{d\theta^2} \widehat{(u_M)}_l(t_n + \theta) + \zeta_l^2 \widehat{(u_M)}_l(t_n + \theta) + \varepsilon^{2} \widehat{f}^n_l(\theta) = 0, \quad \theta \in \mathbb{R},
\label{eq:dtheta}
\end{equation}
where
\begin{equation}
\begin{split}
\zeta_l =\sqrt{1 + \mu^2_l},\quad \widehat{f}^n_l(\theta) =\widehat{(f(u_M))}_l(t_n+\theta).
\end{split}
\label{eq:zeta_n}
\end{equation}

Now, we proceed to apply an exponential wave integrator for solving the ODEs (\ref{eq:dtheta}). By using the variation-of-constants formula for the ODEs (\ref{eq:dtheta}), the general solutions of the above second-order ODEs can be written as
\begin{equation}
\widehat{(u_M)}_l(t_n+\theta) = c^n_l \cos(\theta\zeta_l) +  d^n_l \frac{\sin(\theta\zeta_l)}{\zeta_l}-\frac{\varepsilon^2}{\zeta_l}\int^{\theta}_0 \widehat{f}^n_l(\omega) \sin(\zeta_l(\theta-\omega))d \omega,
\label{eq:voc}
\end{equation}
where $c^n_l$ and $d^n_l$ are two constants to be determined. When $n = 0$, we consider the solution (\ref{eq:voc}) for $\theta \in [0, \tau]$. The initial conditions in the NKGE (\ref{eq:WNE_1D}) imply
\begin{equation}
\widehat{(u_M)}_l(0) = \widehat{\phi}_l,\quad \frac{d}{d t}\widehat{(u_M)}_l(0) = \widehat{\gamma}_l,\quad l \in \mathcal{T}_M.
\label{eq:phihat}
\end{equation}
Plugging (\ref{eq:phihat}) into (\ref{eq:voc}) with $n=0$ and then letting $\theta = \tau$ to determine $c^0_l$ and $d^0_l$, we get
\begin{equation}
\widehat{(u_M)}_l(\tau) = \widehat{\phi}_l \cos(\tau\zeta_l) +  \widehat{\gamma}_l \frac{\sin(\tau\zeta_l)}{\zeta_l} - \frac{\varepsilon^2}{\zeta_l}\int^{\tau}_0 \widehat{f}^0_l(\omega) \sin(\zeta_l(\tau-\omega))d \omega.
\label{eq:uhat0}
\end{equation}
When $n > 0$, we consider the solution (\ref{eq:voc}) for $\theta \in [-\tau, \tau]$ and require the solution to be continuous at $t= t_{n}$ and $t = t_{n-1} = t_n - \tau$. Plugging $\theta = 0$ and $\theta  = -\tau$ into (\ref{eq:voc}) and letting $\theta = \tau$ to determine $c^n_l$ and $d^n_l$, we obtain
\begin{equation}
\begin{split}
\widehat{(u_M)}_l(t_{n+1}) = & -\widehat{(u_M)}_l(t_{n-1}) + 2 \cos(\tau\zeta_l) \widehat{(u_M)}_l(t_n) \\
& - \frac{\varepsilon^2}{\zeta_l}\int^{\tau}_0 \left[ \widehat{f}^n_l(-\omega) +  \widehat{f}^n_l(\omega) \right] \sin(\zeta_l(\tau-\omega))d \omega.
\end{split}
\label{eq:uhatn}
\end{equation}

In order to design an explicit scheme, we approximate the integrals in (\ref{eq:uhat0})-(\ref{eq:uhatn}) by the Gautschi-type quadrature \cite{Gau},
\begin{equation}\nonumber
\begin{split}
& \int^{\tau}_0 \widehat{f}^0_l(\omega) \sin(\zeta_l(\tau-\omega))d \omega \approx  \widehat{f}^0_l(0) \int^{\tau}_0  \sin(\zeta_l(\tau-\omega))d \omega = \frac{\widehat{f}^0_l(0)}{\zeta_l} \left[1-\cos(\tau\zeta_l)\right],\\
& \int^{\tau}_0 \left[ \widehat{f}^n_l(-\omega) + \widehat{f}^n_l(\omega) \right]\sin(\zeta_l(\tau-\omega))d \omega \approx \frac{2\widehat{f}^n_l(0)}{\zeta_l} \left[1-\cos(\tau\zeta_l)\right],\;\; {l \in \mathcal{T}_M},\;\; n\geq 1.
\end{split}
\end{equation}

Denote $\widehat{(u^n_M)}_l$ and $u^n_M(x)$ be the approximations of $\widehat{(u_M)}_l(t_n)$ and $u_M(x, t_n)$, respectively. Choosing $u^0_M(x) = (P_M \phi)(x)$, a Gautschi-type exponential integrator Fourier spectral (EWI-FS) discretization for the NKGE (\ref{eq:WNE_1D}) is to update the numerical approximation $u^{n+1}_M(x) \in X_M (n \geq 0)$ as
\begin{equation}
u^{n+1}_M(x) = \sum_{l \in \mathcal{T}_M} \widehat{(u^{n+1}_M)}_l e^{i\mu_l(x-a)},\quad x \in \overline{\Omega},\quad n\geq 0,
\label{eq:u1}
\end{equation}
where
\begin{equation}
\begin{split}
&\widehat{(u^1_M)}_l = p_l \widehat{\phi}_l + q_l \widehat{\gamma_l} + r_l\widehat{(f(\phi))}_l, \quad l \in \mathcal{T}_M, \\
&\widehat{(u^{n+1}_M)}_l = -\widehat{(u^{n-1}_M)}_l + 2 p_l \widehat{(u^n_M)}_l + 2 r_l \widehat{(f(u^n_M))}_l,\quad{l \in \mathcal{T}_M},\quad n \geq 1,
\end{split}
\label{eq:EWI-FS}
\end{equation}
with the coefficients defined as
\begin{equation}
 p_l = \cos(\tau\zeta_l),\quad q_l = \frac{\sin(\tau\zeta_l)}{\zeta_l},\quad r_l = \frac{\varepsilon^2(\cos(\tau\zeta_l)-1)}{\zeta_l^2}.
\label{eq:pql}
\end{equation}

However, in practice, it is difficult to compute the Fourier coefficients in (\ref{eq:EWI-FS}). We simply replace the projections by the interpolations here, which refers to the Fourier pseudospectral discretizaion. Let $u^n_j$ be the approximation of $u(x_j, t_n)$ and denote $u^0_j = \phi(x_j) (j = 0, 1, \cdots, M)$. For $n = 0, 1, \cdots$, a Gautschi-type exponential integrator Fourier pseudospectral (EWI-FP) discretization for the NKGE (\ref{eq:WNE_1D}) is
\begin{equation}
u^{n+1}_j = \sum_{l \in \mathcal{T}_M}\widetilde{u}^{n+1}_l e^{i\mu_l(x_j-a)},\quad j = 0, 1, \cdots, M,
\label{eq:u_n1}
\end{equation}
where
\begin{equation}
\begin{split}
&\widetilde{u}^1_l = p_l\widetilde{\phi}_l +q_l\widetilde{\gamma_l} + r_l\widetilde{(f(\phi))}_l, \quad l \in \mathcal{T}_M,\\
&\widetilde{u}^{n+1}_l = -\widetilde{u}^{n-1}_l + 2p_l\widetilde{u}^n_l + 2r_l\widetilde{(f(u^{n}))}_l,\quad l \in \mathcal{T}_M, \quad n \geq 1,
\end{split}
\label{eq:EWI-FP}
\end{equation}
with the coefficients $p_l$, $q_l$ and $r_l$ are given in \eqref{eq:pql}.

The EWI-FP \eqref{eq:u_n1}-\eqref{eq:EWI-FP} is explicit, time symmetric and easy to extend to 2D and 3D. The memory cost is $O(M)$ and the computational cost per time step is $O(M\log M)$.

\subsection{Stability analysis}
Let $T_0 > 0$ be a fixed constant and $0 \leq \beta \leq 2$, and denote
\begin{equation}
\sigma_{\max} : = \max_{0 \leq n \leq T_0\varepsilon^{-\beta}/\tau}\|u^n\|^2_{l^{\infty}},
\end{equation}
where $\|u\|_{l^\infty}=\max\limits_{0\leq j\leq M-1}|u_j|$ for $u\in Y_M$.
According to the standard von Neumann stability analysis in \cite{SG}, we can conclude the stability results of the EWI-FP \eqref{eq:u_n1}-\eqref{eq:EWI-FP} for the NKGE \eqref{eq:WNE_1D} in the following lemma.
\begin{lemma} (stability)
The EWI-FP \eqref{eq:u_n1}-\eqref{eq:EWI-FP} is conditionally stable under the stability condition
\begin{equation}
0 < \tau \leq \frac{2h}{\sqrt{\pi^2+h^2(1+\varepsilon^2\sigma_{\max})}},\quad h > 0,\quad 0 < \varepsilon \leq 1.
\label{eq:stability}	
\end{equation}
\label{thm:stability}
\end{lemma}
\begin{proof}
Replacing the nonlinear term by $f(u) = \varepsilon^2\sigma_{\max}u$ and plugging
\begin{equation}\nonumber
u^{n-1}_j = \sum_{l \in \mathcal{T}_M}\widehat{U_l} e^{2ijl\pi/M},\; u^{n}_j =  \sum_{l \in \mathcal{T}_M}\xi_l\widehat{U_l} e^{2ijl\pi/M},\; u^{n+1}_j =  \sum_{l \in \mathcal{T}_M}\xi^2_l\widehat{U_l} e^{2ijl\pi/M},
\end{equation}
into (\ref{eq:EWI-FP}) with $\xi_l$ the amplification factor of the $l$th mode in phase space, we obtain the following characteristic equation
\begin{equation}
\xi^2_l - 2\theta_l\xi_l + 1 = 0,\quad {l \in \mathcal{T}_M},
\label{eq:xi}
\end{equation}
with
\begin{equation*}
\theta_l  =  \cos(\tau\zeta_l) + \frac{\varepsilon^2\sigma_{\max}(\cos(\tau\zeta_l)-1)}{\zeta_l^2} =  1- \left(\frac{2\varepsilon^2\sigma_{\max}}{\zeta_l^2} +2 \right)\sin^2\left(\frac{\tau\zeta_l}{2}\right),\ {l \in \mathcal{T}_M}.
\end{equation*}
Since the characteristic equation (\ref{eq:xi}) implies $\xi_l = \theta_l \pm \sqrt{\theta^2_l-1}$,  it indicates that the stability of the  EWI-FP \eqref{eq:u_n1}-\eqref{eq:EWI-FP} amounts to
\begin{equation}
|\xi_l| \leq 1 \ \Longleftrightarrow \ |\theta_l| \leq 1,\quad {l \in \mathcal{T}_M}.
\end{equation}
Noticing $\sin(x) \leq x$ for $x \geq 0$, under the condition \eqref{eq:stability}, we have
\begin{equation}
0 < 	\left(\frac{2\varepsilon^2\sigma_{\max}}{\zeta_l^2} +2 \right)\sin^2\left(\frac{\tau\zeta_l}{2}\right) \leq \left(\frac{2\varepsilon^2\sigma_{\max}}{\zeta_l^2} +2 \right)\cdot\left(\frac{\tau\zeta_l}{2}\right)^2 \leq 2,
\end{equation}
which immediately leads to the conclusion.
\end{proof}

\begin{remark}
The stability of the EWI-FP \eqref{eq:u_n1}-\eqref{eq:EWI-FP} is related to $\sigma_{\max}$, dependent on the boundedness of the $l^{\infty}$ norm of the numerical solution $u^n$ at the previous time step. The error estimates up to the previous time step could ensure such a bound in the $l^{\infty}$ norm, by making use of the discrete Sobolev inequality, and such an error estimate could be recovered at the next time step, as given by the Theorem \ref{thm:EWI-FS} presented in Section 3. 
\end{remark}

%%%%%%%%%%%%%%%%%%%%%%%%%%%%%
% %    Section 3  Convergence analysis of EWI-FP
%%%%%%%%%%%%%%%%%%%%%%%%%%%%%

\section{Uniform error bounds for long-time dynamics}
In this section, we rigorously carry out the uniform error bounds of the EWI-FS \eqref{eq:u1}-\eqref{eq:EWI-FS}/EWI-FP \eqref{eq:u_n1}-\eqref{eq:EWI-FP} for the NKGE \eqref{eq:WNE_1D} up to the time $t \in [0, T_0/\varepsilon^{\beta}]$ with $0 \leq \beta \leq 2$.
\subsection{Main results}
For an integer $m \geq 0$, $\Omega=(a, b)$, we denote by $H^m(\Omega)$ the standard Sobolev space with norm
\begin{align}
	\|z\|_m^2= \sum\limits_{l\in\mathbb{Z}}(1+\mu_l^2)^m|\widehat{z}_l|^2,\quad \mbox{for}\quad z(x)=\sum\limits_{l\in\mathbb{Z}}\widehat{z}_le^{i\mu_l(x-a)}, \quad \mu_l = \frac{2 \pi l}{b-a},
\end{align}
where $\widehat{z}_l (l\in\mathbb{Z})$ are the Fourier transform coefficients of the function $z(x)$. For $m = 0$, the space is exactly $L^2(\Omega)$ and the corresponding norm is denoted as $\|\cdot\|$. Furthermore, we denote by $H^m_p(\Omega)$ the subspace of $H^m(\Omega)$ which consists of functions with derivatives of the order up to $m-1$ being $(b-a)$-periodic \cite{ST}.

Motivated by the results in \cite{D, KT, K, O} and references therein, we assume that there exists an integer $m_0 \geq 2$ such that the exact solution $u(x, t)$ of the NKGE \eqref{eq:WNE_1D} up to the time $T_{\varepsilon}=T_0/\varepsilon^{\beta}$ with $\beta \in [0, 2]$ and $T_0>0$ fixed satisfies 
\begin{equation*}
{(\rm{A})}\hspace{-1.5cm}
\begin{split}
&u(x, t) \in L^{\infty}\left([0, T_{\varepsilon}]; H^{m_0}_p\right), \ \partial_{t}u(x, t) \in L^{\infty} \left([0, T_{\varepsilon}]; W^{1,4}\right), \\
&\partial_{tt}u(x, t) \in  L^{\infty}\left([0, T_{\varepsilon}]; H^1\right),\\
&\|u(x, t)\|_{L^{\infty}\left([0, T_{\varepsilon}]; H^{m_0}_p\right)} \lesssim 1,\ \ \|\partial_t u(x, t)\|_{L^{\infty}\left([0, T_{\varepsilon}]; W^{1,4}\right)} \lesssim 1,\\
&\|\partial_{tt} u(x, t)\|_{L^{\infty}\left([0, T_{\varepsilon}]; H^1\right)} \lesssim 1.\\
\end{split}
\end{equation*}

Under the assumption (A), we let
\begin{equation}\nonumber
\begin{split}
& M_1 := \max_{\varepsilon \in (0, 1]} \{\|u(x, t)\|_{L^{\infty}\left([0, T_{\varepsilon}]; L^{\infty}\right)} + \|\partial_t u(x, t)\|_{L^{\infty}\left([0, T_{\varepsilon}]; L^{\infty}\right)} \}  \lesssim 1,\\
& M_2 := \sup_{v \neq 0, |v| \leq 1 + M_1} |v|^2 \lesssim 1.
\end{split}
\end{equation}

Assuming
\begin{equation}
\tau \leq  \min\left\{\dfrac18, \;\frac{\pi h}{3\sqrt{\pi^2 + h^2(1+\varepsilon^2 M_2)}}\right\},\quad 0 < \varepsilon \leq 1,
\label{eq:con_tau}
\end{equation}
we can establish the following error bounds of the EWI-FS \eqref{eq:u1}-\eqref{eq:EWI-FS}.\begin{theorem}
Let $u^n_M(x)$ be the approximation obtained from the EWI-FS \eqref{eq:u1}\ -\eqref{eq:EWI-FS}, under the stability condition \eqref{eq:stability} and the assumptions {\rm(A)} and \eqref{eq:con_tau}, there exist constants $h_0 > 0$ and $\tau_0 >0$ sufficiently small and independent of $\varepsilon$, such that for any $0 < \varepsilon \leq 1$ and $0 \leq \beta \leq 2$, when $0 < h \leq h_0$, $0 < \tau \leq \tau_0$, we have
\begin{equation}
\begin{split}
&\|u(x, t_n) - u^n_M(x)\|_{\lambda} \lesssim h^{m_0-\lambda} + \varepsilon^{2-\beta}{\tau^2},\quad \lambda=0,1,\\
& \|u^n_M(x)\|_{L^{\infty}} \leq 1 + M_1, \quad 0 \leq n \leq \frac{T_0/\varepsilon^{\beta}}{\tau}.
\end{split}
\label{eq:error_WNE}
\end{equation}
\label{thm:EWI-FS}
\end{theorem}
\begin{remark}
(1). The EWI-FS \eqref{eq:u1}-\eqref{eq:EWI-FS} is a semi-discretization to the NKGE \eqref{eq:WNE_1D}, while the EWI-FP \eqref{eq:u_n1}-\eqref{eq:EWI-FP} is a full-discretization. The error estimates for the EWI-FP \eqref{eq:u_n1}-\eqref{eq:EWI-FP} are quite similar as those in Theorem \ref{thm:EWI-FS} with the proof being quite similar to that in Section 3.2 and we omit the details for brevity.

(2). In 2D/3D case, Theorem \ref{thm:EWI-FS} is still valid under the technical condition $\tau \lesssim \sqrt{C_d(h)}$, where $C_d(h)=1/|\ln h|$ for $d=2$; and resp., $C_d(h)=h^{1/2}$ for $d=3$. 
\end{remark}

These results are very useful in practical computations on how to select mesh size and time step such that the numerical results are trustable. The error bounds indicate that the $\varepsilon$-scalability of the EWI-FS\eqref{eq:u1}-\eqref{eq:EWI-FS}/EWI-FP \eqref{eq:u_n1}-\eqref{eq:EWI-FP} up to the time at $O(\varepsilon^{-\beta})$ is uniform in terms of $\varepsilon$ and should be taken as: \[h=O(1),\quad \tau=O(1),\quad \mbox{for any} \quad 0<\varepsilon\leq 1\quad \mbox{and}\quad 0\leq\beta\leq 2.\]

\subsection{Proof of Theorem \ref{thm:EWI-FS}}
The key points of the proof are to deal with the nonlinearity and overcome the main difficulty for obtaining the uniform bound of the solution $u_M^n(x)$, i.e., $\|u_M^n(x)\|_{L^\infty}\lesssim 1$. Since the EWI-FS \eqref{eq:u1}-\eqref{eq:EWI-FS} is explicit and the nonlinear term only depends on the previous steps, we adapt the energy method with suitable ``energy" combined with the method of mathematical induction, which is widely used in the literatures \cite{BC, BC1, BD, BDX, BDZ}. The nonlinear part is controlled by the $L^\infty$ norm of the error functions from previous steps by means of the discrete Sobolev inequality and inverse inequality.

The exact solution of the NKGE \eqref{eq:WNE_1D} can be written as
\begin{equation}
u(x, t) = \sum\limits_{l\in\mathbb{Z}} \widehat{u}_l(t) e^{i\mu_l(x-a)},\quad x \in \overline{\Omega},\quad t \geq 0,
\end{equation}
where $\widehat{u}_l(t) (l \in \mathbb{Z})$ are the Fourier transform coefficients of $u(x, t)$. Similar to the derivation of (\ref{eq:uhat0})-(\ref{eq:uhatn}), for $l\in\mathbb{Z}$, we have
\begin{equation}
\widehat{u}_l(\tau) = \widehat{\phi}_l \cos(\tau\zeta_l) +  \widehat{\gamma}_l \frac{\sin(\tau\zeta_l)}{\zeta_l} - \frac{\varepsilon^{2}}{\zeta_l}\int^{\tau}_0 \widehat{F}^0_l(\omega) \sin(\zeta_l(\tau-\omega))d \omega,
\label{eq:u0}
\end{equation}
and for $n\geq 1$,
\begin{equation}
\widehat{u}_l(t_{n+1}) = -\widehat{u}_l(t_{n-1}) + 2 \cos(\tau\zeta_l) \widehat{u}_l(t_n) - \frac{\varepsilon^{2}}{\zeta_l}\int^{\tau}_0 \widehat{(F_+^n)}_l(\omega)  \sin(\zeta_l(\tau-\omega))d \omega,
\label{eq:un}
\end{equation}
where
\begin{equation}
\widehat{(F_+^n)}_l(\omega)=\widehat{F}^n_l(-\omega) +  \widehat{F}^n_l(\omega),\quad \widehat{F}^n_l(\omega) = \widehat{(f(u))}_l(t_n+\omega).
\label{eq:G}
\end{equation}
For $0 \leq n \leq {T_0\varepsilon^{-\beta}}/{\tau}$, denote the ``error'' function
\begin{equation}
\eta^n(x) := P_M u(x, t_n) - u^n_M(x) =\sum_{l \in \mathcal{T}_M}\widehat{\eta}^n_le^{i\mu_l(x-a)},\ x \in \overline{\Omega},
\label{eq:en}
\end{equation}
with
$\widehat{\eta}^n_l = \widehat{u}_l(t_n) - \widehat{(u^n_M)}_l$, $l \in \mathcal{T}_M$. By the assumption (A) and triangle inequality, we have
\begin{equation}\nonumber
\|u(x, t_n) - u^n_M(x)\|_{\lambda} \leq \|u(x, t_n) - P_M u(x, t_n)\|_{\lambda} + \|\eta^n(x)\|_{\lambda} \lesssim h^{m_0-\lambda} + \|\eta^n(x)\|_{{\lambda}}.
\label{eq:eqvi_hat}
\end{equation}
Thus, we only need to estimate $\|\eta^n(x)\|_{\lambda}$ for $0 \leq n \leq {T_0\varepsilon^{-\beta}}/{\tau}$.

Now, we proceed to prove the error bounds in \eqref{eq:error_WNE} by employing the energy method combined with the method of mathematical induction in the following three main steps.

\emph{Step 1. The growth of the ``error" function.} Define the local truncation errors $\xi^{n+1}(x)$ for $0 \leq n \leq T_0\varepsilon^{-\beta}/\tau-1$ as
\begin{equation}\label{eq:xi_co}
\xi^{n+1}(x) =  \sum_{l \in \mathcal{T}_M}\widehat{\xi}^{n+1}_l e^{i\mu_l(x-a)},\end{equation}
with
\begin{align}
&\widehat{\xi}^1_l := \widehat{u}_l(\tau)- p_l\widehat{\phi}_l - q_l\widehat{\gamma_l} - r_l\widehat{(f(\phi))}_l=  -\frac{\varepsilon^{2}}{\zeta_l} \int^{\tau}_0 \widehat{W}_l^1(\omega)\sin(\zeta_l(\tau-\omega))d\omega,\nonumber\\
&\widehat{\xi}^{n+1}_l :=\widehat{u}_l(t_{n+1}) + \widehat{u}_l(t_{n-1})- 2p_l\widehat{u}_l(t_n)- 2r_l\widehat{(f(u))}_l(t_n),\nonumber\\
&\qquad\;=-\frac{\varepsilon^{2}}{\zeta_l} \int^{\tau}_0 \widehat{W}_l^{n+1}(\omega)\sin(\zeta_l(\tau-\omega))d\omega,\quad 1 \leq n \leq {T_0\varepsilon^{-\beta}}/{\tau}-1,\nonumber
\end{align}
where
\begin{equation}
\widehat{W}^{n+1}_l(\omega)=\left\{
\begin{aligned}
&\widehat{F}^0_l(\omega) - \widehat{(f(\phi))}_l,\quad{l \in \mathcal{T}_M},\quad n = 0, \\
&\widehat{(F_+^n)}_l(\omega) - 2\widehat{F}^n_l(0),\quad{l \in \mathcal{T}_M},\quad 1 \leq n \leq {T_0\varepsilon^{-\beta}}/{\tau}-1, \\
\end{aligned}
\right.
\label{eq:w}
\end{equation}
and the coefficients $p_l$, $q_l$ and $r_l$ are given in \eqref{eq:pql}.

For each ${l \in \mathcal{T}_M}$, subtracting \eqref{eq:EWI-FS} from (\ref{eq:u0})-(\ref{eq:un}), we obtain the equation for the ``error" function $\widehat{\eta}^{n+1}_l$ as
\begin{equation}
\begin{split}
&\widehat{\eta}^{n+1}_l = -\widehat{\eta}^{n-1}_l + 2\cos(\tau\zeta_l)\widehat{\eta}^{n}_l + \widehat{\xi}^{n+1}_l+\widehat{\chi}^{n+1}_l, \quad 1 \leq n \leq T_0\varepsilon^{-\beta}/{\tau}-1, \\ 
& \widehat{\eta}^{0}_l = 0, \quad \widehat{\eta}^{1}_l = \widehat{\xi}^1_l,
\end{split}
\label{eq:def_error_n}
\end{equation}
and the nonlinear term errors $\chi^{n+1}(x)\in X_M$ with
\begin{equation}
\widehat{\chi}^{n+1}_l := \frac{2\varepsilon^2(1-\cos(\tau\zeta_l))}{\zeta_l^2}\widehat{V}^{n+1}_l,\quad \widehat{V}^{n+1}_l=\widehat{(f(u^n_M))}_l - \widehat{F}^{n}_l(0).
\label{eq:def_psi}
\end{equation}

\emph{Step 2. Estimates for the cases $n=0$ and $n=1$.}
From the discretization of the initial data, i.e., $u^0_M(x) = P_M\phi(x)$, we have
\begin{equation*}
\|u(x, t= 0) - u^0_M(x)\|_{\lambda} = \|\phi - P_M\phi\|_{\lambda} \lesssim h^{m_0-\lambda},\quad\|u^0_M(x)\|_{L^{\infty}} \leq C h^{m_0-1} + M_1.
\end{equation*}
Therefore, there exists a constant $h_1 > 0$ sufficiently small and independent of $\varepsilon$ such that, when $0 < h \leq h_1$, the error estimates in (\ref{eq:error_WNE}) are valid for $n = 0$. 

Since the calculation for the first step ($n=1$) is different from others, we investigate the first step separately. Under the assumption $\rm{(A)}$, we get
\begin{equation}
\begin{split}
\|\phi^3-u^3(\cdot, \omega)\|^2  = & \int^b_a |u^3(x,0) - u^3(x, \omega)|^2 d x \\
 \leq &\ 9M^2_2\int^b_a |u(x,0) - u(x, \omega)|^2 d x \\
 = &\ 9M^2_2 \int^b_a \left|\int^{\omega}_0 \partial_s u(x, s) ds\right|^2d x \\
\leq &\ 9M^2_2 \int^{b}_{a} \omega \int^{\omega}_0 |\partial_s u(x, s)|^2 ds dx \\% = 9M^4_1
%\omega \int^{\omega}_0 \|\partial_s u(\cdot, s)\|^2 ds \\
\leq &\ 9M^2_2 \omega^2 \|\partial_t u(\cdot, t)\|^2_{L^{\infty}([0, T_{\varepsilon}];L^2)} \\\lesssim & \ \omega^2,\quad 0 \leq \omega \leq \tau.\\
\end{split}
\label{eq:es_f}
\end{equation}
Similarly, we have $\|\phi^3- u^3(\cdot, \omega)\|^2_{1}\lesssim{\omega^2},\ 0 \leq \omega \leq \tau.$

Under the condition \eqref{eq:con_tau}, we get
\begin{equation*}
0 < \tau \zeta_l \leq \frac{\pi}{3},\quad \frac{1}{2}\leq \cos(\zeta_l \tau) < 1, \quad 0 \leq \sin(\zeta_l(\tau - \omega)) \leq \sin(\zeta_l \tau) < 1,\quad 0 \leq \omega \leq \tau.	
\end{equation*}
Noticing $\widehat{\eta}^1_l=\widehat{\xi}^1_l$ and the definition of $\widehat{\xi}^1_l$, by the H\"older inequality, we obtain
\begin{equation}
\begin{split}
\left|\widehat{\eta}^1_l\right|^2 & = \left|\frac{\varepsilon^{2}}{\zeta_l} \int^{\tau}_0 \widehat{W}^1_l(\omega) \sin(\zeta_l(\tau-\omega))d\omega\right|^2 \\
& \leq \varepsilon^4 \int^{\tau}_0 \sin(\zeta_l(\tau-\omega)) d\omega \cdot  \int^{\tau}_0 \left|\widehat{W}^1_l(\omega) \right|^2  \sin(\zeta_l(\tau-\omega)) d\omega\\
& \leq \tau \varepsilon^4 \Big[1-\cos(\zeta_l \tau)\Big] \frac{\sin(\zeta_l \tau)}{\zeta_l \tau}\int^{\tau}_0 \left|\widehat{W}^1_l(\omega) \right|^2 d\omega. \\
& \leq \frac{1}{2}\tau \varepsilon^{4} \int^{\tau}_0 \left|\widehat{W}^1_l(\omega) \right|^2 d\omega.
\end{split}
\label{eq:eta_n1}
\end{equation}
Multiplying both sides of the above equalities by $(1+\mu_l^2)^{\lambda}$ and then summing up for ${l \in \mathcal{T}_M}$, the Parseval's identity equality, triangle inequality, \eqref{eq:w} and \eqref{eq:es_f} lead to
\begin{align*}
\|\eta^1(x)\|^2_{\lambda} & \leq \frac{1}{2} \tau {\varepsilon^{4}} \sum_{l \in \mathcal{T}_M} \left(1+\mu_l^{2}\right) ^{\lambda}\int^{\tau}_0 \left|\widehat{W}^1_l(\omega) \right|^2 d\omega \\
&= \frac{1}{2} \tau {\varepsilon^{4}}\sum_{l \in \mathcal{T}_M} \left(1+\mu_l^{2}\right) ^{\lambda} \int^{\tau}_0 \left| \widehat{(f(u))}_l(\omega)-\widehat{(f(\phi))}_l \right|^2 d\omega \\
&\lesssim \tau {\varepsilon^{4}}\int^{\tau}_0 \|u^3(\cdot, \omega) - \phi^3\|^2_{\lambda}d\omega\\%\label{eq:es_xi1}
&\lesssim \tau^4 {\varepsilon^{4}},\quad \lambda=0,1.
\end{align*}
Thus, we immediately can obtain
\begin{equation}\nonumber
\|u(x, t_1) - u^1_M(x)\|_{\lambda}\leq \|u(x, t_1) - P_Mu(x, t_1)\|_{\lambda}+ \|\eta^1(x)\|_{\lambda}  \lesssim h^{m_0-\lambda}+\varepsilon^{2}{\tau^2}.
\end{equation}
By the triangle inequality and inverse inequality, there exist $h_2> 0$ and $\tau_1 > 0$ sufficiently small such that when $0 < h \leq h_2$ and $0 < \tau \leq \tau_1$, we have
\begin{equation}
\|u^1_M(x)\|_{L^{\infty}} \leq 1 + M_1.
\label{eq:u_infty}
\end{equation}
Therefore, the estimates in (\ref{eq:error_WNE}) are valid when $n = 1$.

\emph{Step 3. Estimates for the cases $2\leq n \leq T_0\varepsilon^{-\beta}/\tau$.} Assume that the estimates in (\ref{eq:error_WNE}) are valid for all $1 \leq n \leq m \leq T_0\varepsilon^{-\beta}/\tau - 1$, then we need to prove that they are still valid when $n = m+1$.

On the one hand, under the assumption $\rm{(A)}$, by the H\"older inequality, we have
\begin{equation}
\begin{split}	
&\|2u^3(\cdot, t_n)-u^3(\cdot, t_n-\omega)-u^3(\cdot, t_n+\omega)\|^2 \\
\leq &  \int^b_a \left|\int^{\omega}_0 \int^s_{-s} \partial_{\theta\theta} u^3(x, t_n + \theta)d\theta ds \right|^2 dx \\
\leq & \int^b_a \omega \int^{\omega}_0 2 s   \int^s_{-s} \left|\partial_{\theta\theta} u^3(x, t_n + \theta)\right|^2 d\theta ds dx \\
\leq &\int^{\omega}_0 2\omega s \int^s_{-s}\left(9M^2_2\|\partial_{\theta\theta} u(\cdot, t_n + \theta)\|^2 + 36M_2\|\partial_{\theta} u(\cdot, t_n + \theta)\|^4_{L^4}\right)d\theta ds \\
\lesssim &\  \omega^4 \left[\|\partial_{t} u(\cdot, t)\|^4_{L^{\infty}([0, T_{\varepsilon}]; L^4)}+\|\partial_{tt} u(\cdot, t)\|^2_{L^{\infty}([0, T_{\varepsilon}]; L^2)}\right] \\
\lesssim & \ \tau^4, \quad 0 \leq \omega \leq \tau.
\end{split}
\label{eq:fn}
\end{equation}
Similarly, we have  $\| u^3 (\cdot, t_n-\omega) + u^3(\cdot, t_n+\omega) - 2u^3(\cdot, t_n)\|^2_{1}\lesssim{\tau^4}, \ 0 \leq \omega \leq \tau$.

On the other hand, noticing the definitions of $\widehat{\xi}^{n+1}_l$ and $\widehat{\chi}^{n+1}_l$, similar to \eqref{eq:eta_n1}, we have
\begin{equation}
\begin{split}
\left|\widehat{\xi}^{n+1}_l\right|^2 & = \left|\frac{\varepsilon^{2}}{\zeta_l} \int^{\tau}_0 \widehat{W}^{n+1}_l(\omega) \sin(\zeta_l(\tau-\omega))d\omega\right|^2 \\
& \lesssim \tau {\varepsilon^{4}}\Big[1-\cos(\tau\zeta_l)\Big]\int^{\tau}_0 \left|\widehat{W}^{n+1}_l(\omega) \right|^2 d\omega,\quad {l \in \mathcal{T}_M},\quad 1 \leq n \leq m,\\	
\end{split}	
\label{eq:xi_n1}
\end{equation}
\begin{equation}
\begin{split}
|\widehat{\chi}^{n+1}_l|^2 &=\left|\frac{2\varepsilon^2(1-\cos(\tau\zeta_l))}{\zeta_l^2}\widehat{V}^{n+1}_l\right|^2 \\
& = \left|\frac{2\varepsilon^{2}}{\zeta_l} \widehat{V}^{n+1}_l  \int^{\tau}_0\sin(\zeta_l(\tau-\omega))d\omega\right|^2 \\
& \lesssim \tau^2 {\varepsilon^{4}}\Big[1-\cos(\tau\zeta_l)\Big] \left|\widehat{V}^{n+1}_l \right|^2,\quad l \in \mathcal{T}_M,\quad 1 \leq n \leq m.	
\end{split}
\label{eq:chi_n1}
\end{equation}

Multiplying the above inequalities with $(1+\mu^2_l)^{\lambda} (\lambda = 0, 1)$ , dividing them by $1-\cos(\tau\zeta_l)$ and then summing up ${l \in \mathcal{T}_M}$, for $1\leq n\leq m$, the Parseval's identity equality, triangle inequality,\eqref{eq:w}, \eqref{eq:def_psi} and \eqref{eq:fn} lead to
\begin{equation}
\begin{split}	
&\sum\limits_{l \in \mathcal{T}_M}\frac{(1+\mu_l^2)^{\lambda}}{1-\cos(\tau\zeta_l)} \left|\widehat{\xi}^{n+1}_l\right|^2 \\%& = (b_{\varepsilon}-a)\sum^{M/2-1}_{l=-M/2}(1+\mu^2_l) |\widehat{e}^1_l|^2  = (b_{\varepsilon}-a)\sum^{M/2-1}_{l=-M/2} (1+\mu^2_l)|\widehat{\xi}^1_l|^2 \\
\lesssim & \ \tau {\varepsilon^{4}} \sum_{l \in \mathcal{T}_M} \left(1+\mu_l^{2}\right)^{\lambda}\int^{\tau}_0 \left|\widehat{W}^{n+1}_l(\omega) \right|^2 d\omega \\
\lesssim &\ \tau {\varepsilon^{4}} \int^{\tau}_0  \|u^3(\cdot, t_n-\omega) + u^3(\cdot, t_n+\omega) - 2u^3(\cdot, t_n)\|^2_{\lambda}d\omega \\
%&\lesssim \tau \varepsilon^{4-2\beta} \int^{\tau}_0 \frac{\omega^2}{\varepsilon^{2\beta}} d\omega \lesssim  \varepsilon^{4-4\beta}{\tau^4}.
\lesssim &\   \tau^6 {\varepsilon^{4}},
\end{split}
 \label{eq:xi_sn1}
 \end{equation}
\begin{equation}
\begin{split}
\sum\limits_{l \in \mathcal{T}_M}\frac{(1+\mu_l^2)^{\lambda}}{1-\cos(\tau\zeta_l)}\left|\widehat{\chi}^{n+1}_l\right|^2 \lesssim & \  \tau^2 {\varepsilon^4} \sum_{l \in \mathcal{T}_M}(1+\mu_l^2)^{\lambda}\left|\widehat{V}^{n+1}_l \right|^2 \\
\lesssim &\ \tau^2 {\varepsilon^4} \|u^3(\cdot, t_n) - (u^n_M)^3\|^2_{\lambda} \\
 \lesssim &\  \tau^2 {\varepsilon^4}M^2_2\|u(\cdot, t_n) - u^n_M\|^2_{\lambda} \\
\lesssim &\ \tau^2\varepsilon^4\left(h^{m_0-\lambda}+\varepsilon^{2-\beta}\tau^2\right)^2.
\end{split}
\label{eq:chi_sn1}
\end{equation}

Define the ``energy'' function as
\begin{equation}
\mathcal{E}^n = \sum_{l \in \mathcal{T}_M} \widehat{\mathcal{E}}^n_l,\;\widehat{\mathcal{E}}^n_l =(1+\mu_l^2)^{\lambda}\left[\left|\widehat{\eta}^{n}_l\right|^2+\left|\widehat{\eta}^{n+1}_l\right|^2  + \frac{\cos(\tau\zeta_l)}{1-\cos(\tau\zeta_l)}\left|\widehat{\eta}^{n+1}_l -\widehat{\eta}^{n}_l\right|^2\right].
\label{eq:def_E}
\end{equation}
For $n = 0$, we have 
\begin{equation*}
\mathcal{E}^0  =  \sum_{l \in \mathcal{T}_M} \frac{(1+\mu_l^2)^{\lambda}}{1-\cos(\tau\zeta_l)} \left|\widehat{\eta}^1_l\right|^2 \leq \tau {\varepsilon^{4}} \sum_{l \in \mathcal{T}_M} \left(1+\mu_l^{2}\right) ^{\lambda}\int^{\tau}_0 \left|\widehat{W}^1_l(\omega) \right|^2 d\omega \\
 \lesssim \tau^4\varepsilon^{4}.
\end{equation*}
Noticing $ 0 \leq \beta \leq 2$, multiplying both sides of \eqref{eq:def_error_n} by $\left(1+\mu^{2}_l\right)^{\lambda}\left(\widehat{\eta}^{n+1}_l\right.-$ $\left.\widehat{\eta}^{n-1}_l\right)$,  dividing it by $1-\cos(\tau\zeta_l)$ and summing up for ${l \in \mathcal{T}_M}$, the Young's inequality and \eqref{eq:xi_n1}-\eqref{eq:chi_sn1} result in
\begin{align*}
{\mathcal{E}^n} - {\mathcal{E}^{n-1}}  \leq \ & \sum_{l \in \mathcal{T}_M}\frac{(1+\mu_l^2)^{\lambda}}{1-\cos(\tau\zeta_l)}\left|\widehat{\xi}^{n+1}_l+\widehat{\chi}^{n+1}_l\right|\cdot \left|\widehat{\eta}^{n+1}_l - \widehat{\eta}^{n-1}_l\right|\\
 \leq & \sum_{l \in \mathcal{T}_M}\frac{(1+\mu_l^2)^{\lambda}}{1-\cos(\tau\zeta_l)}\left(2\varepsilon^{\beta}\tau \left|\widehat{\eta}^{n+1}_l - \widehat{\eta}^{n}_l\right|^2 + 2\varepsilon^{\beta}\tau \left|\widehat{\eta}^{n}_l - \widehat{\eta}^{n-1}_l\right|^2 \right.\\
 & \left.\quad\quad\quad\quad\quad \quad\quad \quad\ \ +\frac{1}{\varepsilon^{\beta}\tau}\left|\widehat{\xi}^{n+1}_l+\widehat{\chi}^{n+1}_l\right|^2\right)\\
\leq &  \sum_{l \in \mathcal{T}_M}\frac{4\varepsilon^{\beta}\tau\cos(\tau\zeta_l)}{1-\cos(\tau\zeta_l)}(1+\mu_l^2)^{\lambda}\left(\left|\widehat{\eta}^{n+1}_l - \widehat{\eta}^{n}_l\right|^2 + \left|\widehat{\eta}^{n}_l - \widehat{\eta}^{n-1}_l\right|^2 \right)\\
&+ \sum_{l \in \mathcal{T}_M}\frac{2(1+\mu_l^2)^{\lambda}}{\varepsilon^{\beta}\tau(1-\cos(\tau\zeta_l))}\left(\left|\widehat{\xi}^{n+1}_l\right|^2 + \left|\widehat{\chi}^{n+1}_l\right|^2\right)\\
\leq & \ 4\varepsilon^{\beta}\tau\left(\mathcal{E}^n+\mathcal{E}^{n-1}\right) + C\varepsilon^{4-\beta}\tau\left(h^{m_0-\lambda}+ \tau^2\right)^2,\quad  1 \leq n \leq m,
\end{align*}
where the constant $C$ is independent of $h$, $\tau$ and $\varepsilon$. Summing the above inequality for $n = 1, 2, \cdots, m$, and noticing the condition $\tau \leq 1/8$, we get
\begin{equation}
\mathcal{E}^m \lesssim \mathcal{E}^0 + \varepsilon^{\beta} \tau \sum^{m-1}_{n=0} \mathcal{E}^n +   T_0\varepsilon^{4-2\beta} \left(h^{m_0-\lambda}+ \tau^2\right)^2, \quad 1 \leq m\leq T_0\varepsilon^{-\beta}/\tau - 1.
\end{equation}
Hence, the Gronwall's inequality suggests that there exists a constant $\tau_2 > 0$ sufficiently small, such that when $0\leq\tau\leq \tau_2$, the following holds for $1 \leq m\leq T_0\varepsilon^{-\beta}/\tau - 1$,
\begin{equation}
\mathcal{E}^m \lesssim \mathcal{E}^0 +  \varepsilon^{4-2\beta}\left(h^{m_0-\lambda} + \tau^2\right)^2\lesssim\varepsilon^{4-2\beta}\left(h^{m_0-\lambda}+\tau^2\right)^2.
\label{eq:Em}
\end{equation}
Recalling the definition of $\mathcal{E}^m$ in \eqref{eq:def_E},  for $1 \leq m\leq T_0\varepsilon^{-\beta}/\tau - 1$, we can obtain the error estimate
\begin{equation*}
\|\eta^{m+1}\|^2_{\lambda}  = \sum_{l \in \mathcal{T}_M}(1+\mu^{2}_l)^{\lambda} \left|\widehat{\eta}^{m+1}_l\right|^2 \leq  \mathcal{E}^m\lesssim \varepsilon^{4-2\beta}\left(h^{m_0-\lambda}+\tau^2\right)^2,
\end{equation*}
by combining \eqref{eq:def_E} with the Parseval's identity equality and \eqref{eq:Em}, which immediately concludes that the first inequality in \eqref{eq:error_WNE} is valid for $n = m+1$.

Lastly, we have to prove the error estimate of $\|u^{m+1}_M(x)\|_{L^{\infty}}$ for $1\leq m\leq T_0\varepsilon^{-\beta}/\tau-1$. In fact, the inverse inequality and triangle inequality will imply that
there exist $h_3 > 0$ and  $\tau_3 > 0$ sufficiently small such that when $0 < h \leq h_3$ and $0 < \tau \leq \tau_3$, we have
\begin{equation*}
\|u^{m+1}_M(x)\|_{L^{\infty}} \leq \|u(x, t_{m+1})-u^{m+1}_M(x)\|_{L^{\infty}} +\|u(x, t_{m+1})\|_{L^\infty}\leq 1 + M_1.
\end{equation*}
Overall, under the choice of $h_0 = \min\{h_1, h_2, h_3\}$ and $\tau_0 = \min\{\tau_1, \tau_2,\tau_3\}$, the proof of Theorem \ref{thm:EWI-FS} is completed by the method of mathematical induction.

%%%%%%%%%%%%%%%%%%%%%%%%%%%%%
% %    Section 4  Numerical results
%%%%%%%%%%%%%%%%%%%%%%%%%%%%%
\section{Numerical results}
In this section, we present numerical results  of the EWI-FP \eqref{eq:u_n1}-\eqref{eq:EWI-FP} for the NKGE \eqref{eq:WNE_1D} with weak nonlinearity to support our error estimates. In our numerical experiments, we choose the initial data
\begin{equation}
\phi(x) = \frac{1}{2+\cos^2(x)} \quad  \mbox{and} \quad \gamma(x) = \sin(x), \quad x \in [0 ,2\pi].
\label{eq:long_initial}
\end{equation}
The numerical computation is carried out on a time interval $[0, T_0/\varepsilon^{\beta}]$ with $0 \leq \beta \leq 2$ and $T_0 > 0$ fixed. Here, we study the following three cases with respect to different $\beta$:

Case I. Fixed time dynamics up to the time at $O(1)$, i.e., $\beta = 0$;

Case II. Intermediate long-time dynamics up to the time at  $O(\varepsilon^{-1})$, i.e., $\beta = 1$;

Case III. Long-time dynamics up to the time at $O(\varepsilon^{-2})$, i.e., $\beta = 2$.

The ``exact" solution $u(x, t)$ is computed by the time-splitting Fourier pseudospectral method \cite{BFSU,DXZ} with a very fine mesh size $h_e = \pi/32$ and a very small time step $\tau_e = 5\times10^{-4}$. Denote $u^n_{h, \tau}$ as the numerical solution at $t=t_n$ by the EWI-FP \eqref{eq:u_n1}-\eqref{eq:EWI-FP} with mesh size $h$ and time step $\tau$. The errors are denoted as $e(x, t_n) \in X_M$ with $e(x, t_n) = u(x, t_n) - I_M(u^n_{h, \tau})(x)$. In order to quantify the numerical results, we measure the $H^1$ norm of $e(x, t_n)$.

The errors are displayed at $t = 1/\varepsilon^{\beta}$ with different $\varepsilon$ and $\beta$. In order to test the spatial errors, we fix the time step as $\tau = 5\times10^{-4}$ such that the temporal error can be ignored and solve the NKGE \eqref{eq:WNE_1D} under different mesh size $h$. Table \ref{tab:WNE_h} depicts the spatial errors for $\beta = 0$, $\beta = 1$ and $\beta = 2$. Then, we check the temporal errors for different $0 \leq \varepsilon \leq 1$ and $0 \leq \beta \leq 2$ with different time step $\tau$ and a very fine mesh size $h = \pi/32$ such that the spatial errors can be neglected. Tables \ref{tab:WNE_beta0_t}-\ref{tab:WNE_beta2_t} show the temporal errors for $\beta = 0$, $\beta = 1$ and $\beta = 2$, respectively.

\begin{table}
\caption{Spatial errors of the EWI-FP \eqref{eq:u_n1}-\eqref{eq:EWI-FP} for the NKGE \eqref{eq:WNE_1D} with \eqref{eq:long_initial} for different $\beta$ and $\varepsilon$.}
\centering
\begin{tabular}{c|ccccc}
\hline
&$\|e(\cdot,1/\varepsilon^{\beta})\|_{1}$ &\;\;\; $h_0 = \pi/2 $\;\;\; & \;\;\;$h_0/2 $ \;\;\;& \;\;\;$h_0/2^2 $ \;\;\; &\;\;\;$h_0/2^3$ \;\;\; \\
\hline
\multirow{5}{*}{$\beta=0$}
& $\varepsilon_0 = 1$ & 4.05E-2 & 8.80E-3 & 1.53E-4 & 7.19E-8 \\
&$\varepsilon_0 / 2$   & 4.78E-2 & 8.48E-3 & 1.58E-4 & 2.37E-8 \\
&$\varepsilon_0 / 2^2$ & 5.17E-2 & 8.36E-3 & 1.59E-4 & 1.15E-8 \\
&$\varepsilon_0 / 2^3$ & 5.28E-2 & 8.33E-3 & 1.59E-4 & 1.00E-8 \\
&$\varepsilon_0 / 2^4$ & 5.31E-2 & 8.32E-3 & 1.59E-4 & 9.89E-9 \\
\hline
\multirow{5}{*}{$\beta=1$}
& $\varepsilon_0 = 1$ & 4.05E-2 & 8.80E-3 & 1.53E-4 & 7.19E-8 \\
&$\varepsilon_0 / 2$ & 3.98E-2 & 6.27E-3 & 5.61E-5 & 4.19E-8 \\
&$\varepsilon_0 / 2^2$ & 1.57E-2 & 8.14E-3 & 1.33E-4 & 4.03E-8 \\
&$\varepsilon_0 / 2^3$ & 1.02E-2 & 3.17E-3 & 2.82E-5 & 1.08E-8 \\
&$\varepsilon_0 / 2^4$ & 6.08E-3 & 3.44E-3 & 1.41E-5 & 1.98E-8 \\
\hline
\multirow{5}{*}{$\beta=2$}
& $\varepsilon_0 = 1$ & 4.05E-2 & 8.80E-3 & 1.53E-4 & 7.19E-8 \\
&$\varepsilon_0 / 2$   & 4.04E-2 & 8.46E-3 & 1.40E-4 & 9.30E-8 \\
&$\varepsilon_0 / 2^2$ & 6.12E-2 & 4.18E-3 & 1.57E-5 & 6.90E-8 \\
&$\varepsilon_0 / 2^3$ & 1.01E-1 & 3.25E-3 & 1.45E-4 & 1.35E-7 \\
&$\varepsilon_0 / 2^4$ & 6.05E-2 & 1.31E-3  & 1.34E-4 & 4.16E-7 \\
\hline
\end{tabular}
\label{tab:WNE_h}
\end{table}

\begin{table}
\caption{Temporal errors of the EWI-FP \eqref{eq:u_n1}-\eqref{eq:EWI-FP} for the NKGE \eqref{eq:WNE_1D} with \eqref{eq:long_initial} and $\beta = 0$.}
\centering
\begin{tabular}{ccccccc}
\hline
$\|e(\cdot,1/\varepsilon^{\beta})\|_{1}$ &$\tau_0 = 0.2 $ & $\tau_0/2 $ &$\tau_0/2^2 $ & $\tau_0/2^3$ & $\tau_0/2^4$  & $\tau_0/2^5$ \\ 
\hline
$\varepsilon_0 = 1$ & 4.59E-2 & 1.13E-2 & 2.82E-3 & 7.04E-4 & 1.76E-4 & 4.37E-5 \\
order & - & 2.02 & 2.00 & 2.00 & 2.00 & 2.01 \\
\hline
$\varepsilon_0 / 2 $ & 1.48E-2 & 3.66E-3 & 9.11E-4 & 2.27E-4 & 5.68E-5 & 1.41E-5 \\
order & -  & 2.02 & 2.01 & 2.00 & 2.00 & 2.01 \\
\hline
$\varepsilon_0 / 2^2 $ & 4.05E-3 & 1.00E-3 & 2.49E-4 & 6.23E-5 & 1.55E-5 & 3.86E-6 \\
order & -  & 2.02 & 2.01 & 2.00 & 2.01 & 2.01 \\
\hline
$\varepsilon_0 / 2^3 $ & 1.04E-3 & 2.56E-4 & 6.39E-5 & 1.59E-5 & 3.98E-6 & 9.89E-7 \\
order & -  & 2.02 & 2.00 & 2.01 & 2.00 & 2.01 \\
\hline
$\varepsilon_0 / 2^4 $ & 2.61E-4 & 6.44E-5 & 1.61E-5 & 4.01E-6 & 1.00E-6 & 2.49E-7 \\
order & -  & 2.02 & 2.00 & 2.01 & 2.00 & 2.01 \\
\hline
$\varepsilon_0 / 2^5 $ & 6.53E-5 & 1.61E-5 & 4.02E-6 & 1.00E-6 & 2.51E-7 & 6.23E-8 \\
order & -  & 2.02 & 2.00 & 2.01 & 1.99 & 2.01 \\
\hline
\end{tabular}
\label{tab:WNE_beta0_t}
\end{table}

\begin{table}
\caption{Temporal errors of the EWI-FP \eqref{eq:u_n1}-\eqref{eq:EWI-FP} for the NKGE \eqref{eq:WNE_1D} with \eqref{eq:long_initial} and $\beta = 1$.}
\centering
\begin{tabular}{ccccccc}
\hline
$\|e(\cdot,1/\varepsilon^{\beta})\|_{1}$ &$\tau_0 = 0.2 $ & $\tau_0/2 $ &$\tau_0/2^2 $ & $\tau_0/2^3$ & $\tau_0/2^4$  & $\tau_0/2^5$ \\
\hline
$\varepsilon_0 = 1$ & 4.59E-2 & 1.13E-2 & 2.82E-3 & 7.04E-4 & 1.76E-4 & 4.37E-5 \\
order & - & 2.02 & 2.00 & 2.00 & 2.00 & 2.01 \\
\hline
$\varepsilon_0 / 2 $ & 1.30E-2 & 3.22E-3 & 8.04E-4 & 2.01E-4 & 5.02E-5 & 1.25E-5 \\
order & -  & 2.01 & 2.00 & 2.00 & 2.00 & 2.01 \\
\hline
$\varepsilon_0 / 2^2 $ & 5.76E-3 & 1.43E-3 & 3.56E-4 & 8.90E-5 & 2.23E-5 & 5.57E-6 \\
order & -  & 2.01 & 2.01 & 2.00 & 2.00 & 2.00 \\
\hline
$\varepsilon_0 / 2^3 $ & 2.30E-3 & 5.72E-4 & 1.43E-4 & 3.57E-5 & 8.92E-6 & 2.23E-6 \\
order & -  & 2.01 & 2.00 & 2.00 & 2.00 & 2.00 \\
\hline
$\varepsilon_0 / 2^4 $ & 1.66E-3 & 4.11E-4 & 1.03E-4 & 2.56E-5 & 6.41E-6 & 1.60E-6 \\
order & -  & 2.01 & 2.00 & 2.01 & 2.00 & 2.00 \\
\hline
$\varepsilon_0 / 2^5 $ & 4.18E-4 & 1.04E-4 & 2.59E-5 & 6.48E-6 & 1.62E-6 & 4.05E-7 \\
order & -  & 2.01 & 2.01 & 2.00 & 2.00 & 2.00 \\ 
\hline
\end{tabular}
\label{tab:WNE_beta1_t}
\end{table}

\begin{table}
\caption{Temporal errors of the EWI-FP \eqref{eq:u_n1}-\eqref{eq:EWI-FP} for the NKGE \eqref{eq:WNE_1D} with \eqref{eq:long_initial} and $\beta = 2$.}
\centering
\begin{tabular}{ccccccc}
\hline
$\|e(\cdot,1/\varepsilon^{\beta})\|_{1}$ &$\tau_0 = 0.2 $ & $\tau_0/2 $ &$\tau_0/2^2 $ & $\tau_0/2^3$ & $\tau_0/2^4$  & $\tau_0/2^5$ \\
\hline
$\varepsilon_0 = 1$ & 4.59E-2 & 1.13E-2 & 2.82E-3 & 7.04E-4 & 1.76E-4 & 4.37E-5 \\
order & - & 2.02 & 2.00 & 2.00 & 2.00 & 2.01 \\
\hline
$\varepsilon_0 / 2 $ & 3.17E-2 & 7.83E-3 & 1.95E-3 & 4.88E-4 & 1.22E-4 & 3.04E-5 \\
order & -  & 2.02 & 2.01 & 2.00 & 2.00 & 2.00 \\
\hline
$\varepsilon_0 / 2^2 $ &2.51E-2 & 6.23E-3 & 1.55E-3 & 3.88E-4 & 9.70E-5 & 2.42E-5 \\
order & -  & 2.01 & 2.01 & 2.00 & 2.00 & 2.00 \\
\hline
$\varepsilon_0 / 2^3 $ & 3.28E-2 & 8.14E-3 & 2.03E-3 & 5.08E-4 & 1.27E-4 & 3.17E-5 \\
order & -  & 2.01 & 2.00 & 2.00 & 2.00 & 2.00 \\
\hline
$\varepsilon_0 / 2^4 $ & 2.50E-2 & 6.23E-3 & 1.56E-3 & 3.89E-4 & 9.72E-5 & 2.43E-5 \\
order & -  & 2.00 & 2.00 & 2.00 & 2.00 & 2.00 \\
\hline
$\varepsilon_0 / 2^5 $ & 2.88E-2 & 7.17E-3 & 1.79E-3 & 4.47E-4 & 1.12E-4 & 2.79E-5 \\
order & -  & 2.01 & 2.00 & 2.00 & 2.00 & 2.01 \\
\hline
\end{tabular}
\label{tab:WNE_beta2_t}
\end{table}

From Tables \ref{tab:WNE_h}-\ref{tab:WNE_beta2_t} and additional similar numerical results not shown here for brevity, we can draw the following observations on the EWI-FP \eqref{eq:u_n1}-\eqref{eq:EWI-FP} for the NKGE \eqref{eq:WNE_1D} up to the time at $O(\varepsilon^{-\beta})$ with $0\leq\beta\leq 2$:

(i) In space, the EWI-FP \eqref{eq:u_n1}-\eqref{eq:EWI-FP} is uniformly spectral accurate for any $0 < \varepsilon \leq 1$ and $0 \leq \beta \leq 2$ (cf. each row in Table \ref{tab:WNE_h}) and the spatial errors are almost independent of $\varepsilon$ (cf. each column in Table \ref{tab:WNE_h}).

(ii) In time, for any fixed $\varepsilon = \varepsilon_0>0$, the EWI-FP \eqref{eq:u_n1}-\eqref{eq:EWI-FP} is uniformly second-order accurate (cf. the first rows in Tables \ref{tab:WNE_beta0_t}-\ref{tab:WNE_beta2_t}), which agree with those results in the literature. In addition, Tables \ref{tab:WNE_beta0_t}-\ref{tab:WNE_beta2_t} illustrate that the error bounds of temporal discretization for the EWI-FP \eqref{eq:u_n1}-\eqref{eq:EWI-FP} uniformly behave like $O(\varepsilon^2\tau^2)$ for the fixed time dynamics up to the time at $O(1)$, i.e., $\beta=0$ (cf. each row and column in Table \ref{tab:WNE_beta0_t}), and  $O(\varepsilon\tau^2)$ for the intermediate long-time dynamics up to the time at $O(\varepsilon^{-1})$, i.e., $\beta=1$ (cf. each row and column in Table \ref{tab:WNE_beta1_t}), and resp. $O(\tau^2)$ for the long-time dynamics up to the time at $O(\varepsilon^{-2})$, i.e., $\beta=2$ (cf. each row and column in Table \ref{tab:WNE_beta2_t}). In summary, our numerical results confirm the error bounds in Theorem \ref{thm:EWI-FS} and demonstrate that they are sharp.

%%%%%%%%%%%%%%%%%%%%%%%%%%%%%
% %    Section 5  Extension to an oscillatory NKGE
%%%%%%%%%%%%%%%%%%%%%%%%%%%%%
\section{Extension to an oscillatory NKGE}
By a rescaling in time $s = \varepsilon^{\beta}t$ with $0\leq \beta \leq 2$ and denoting $v({\bf{x}}, s) := u({\bf{x}}, s/\varepsilon^{\beta}) = u({\bf{x}}, t)$ in the NKGE \eqref{eq:WNE}, we can reformulate the NKGE \eqref{eq:WNE} into the following oscillatory NKGE
\begin{equation}
\left\{
\begin{aligned}
&\varepsilon^{2\beta}\partial_{ss} v({\bf x}, s) - \Delta v({\bf x}, s) + v({\bf x}, s) + \varepsilon^2 v^3({\bf x}, s) = 0,\quad x \in \mathbb{T}^d ,\quad s >0,\\
&v({\bf x} , 0) = \phi({\bf x}),\quad \partial_s v({\bf x}, 0) = \varepsilon^{-\beta}\gamma({\bf x}),\quad x \in \mathbb{T}^d,
\end{aligned}\right.
\label{eq:HOE}
\end{equation}
which is also time symmetric or time reversible and conserves the energy, i.e.,
\begin{equation*}
\begin{split}
\mathcal{E}(s) &:= \int_{{\mathbb{T}}^d} \left[ \varepsilon^{2\beta}|\partial_s v ({\bf{x}}, s)|^2 + |\nabla v({\bf{x}}, s)|^2 + |v({\bf{x}}, s)|^2 +\frac{\varepsilon^2}{2}|v({\bf{x}}, s)|^4  \right] d {\bf{x}} \\
& \equiv  \int_{{\mathbb{T}}^d} \left[ |\gamma({\bf{x}})|^2 + |\nabla \phi({\bf{x}})|^2 + |\phi({\bf{x}})|^2 + \frac{\varepsilon^2}{2}|\phi({\bf{x}})|^4 \right] d {\bf{x}} \\
& = E(0) = O(1), \qquad s \geq 0.
\end{split}
\end{equation*}

Recently, more attentions are paid to the properties of the highly oscillatory evolution equations \cite{BD, BDZ, BFS, CCLMV, FS}. The solution of the NKGE \eqref{eq:HOE} propagates waves with wavelength at $O(1)$ in space and $O(\varepsilon^{\beta})$ in time, respectively \cite{BFY}. To illustrate this, Figures \ref{fig:vt}-\ref{fig:vx} show the solutions $v(\pi, s)$ and $v(x,1)$ of the oscillatory NKGE \eqref{eq:HOE} with $d=1$, $\mathbb{T}^1=(0,2\pi)$ and initial data \eqref{eq:long_initial} for different $0<\varepsilon\le 1$ and $\beta$. It is important to point out that the temporal oscillation frequency depends on the space Fourier variable and the wave speed of the oscillatory NKGE \eqref{eq:HOE} is at $O(\varepsilon^{-\beta})$, which is quite different from that of the NKGE in the nonrelativistic limit regime considered in \cite{BD, BDZ, BFS, FS}. In fact, in the nonrelativistic limit regime, the solution of the NKGE  propagates waves with wavelength at $O(1)$ in space and $O(\varepsilon^2)$ in time, and wave speed in space at $O(1)$.

%\vspace{-0.5in}
\begin{figure}[ht!]
\begin{minipage}{0.48\textwidth}
\centerline{\includegraphics[width=6.8cm,height=4.5cm]{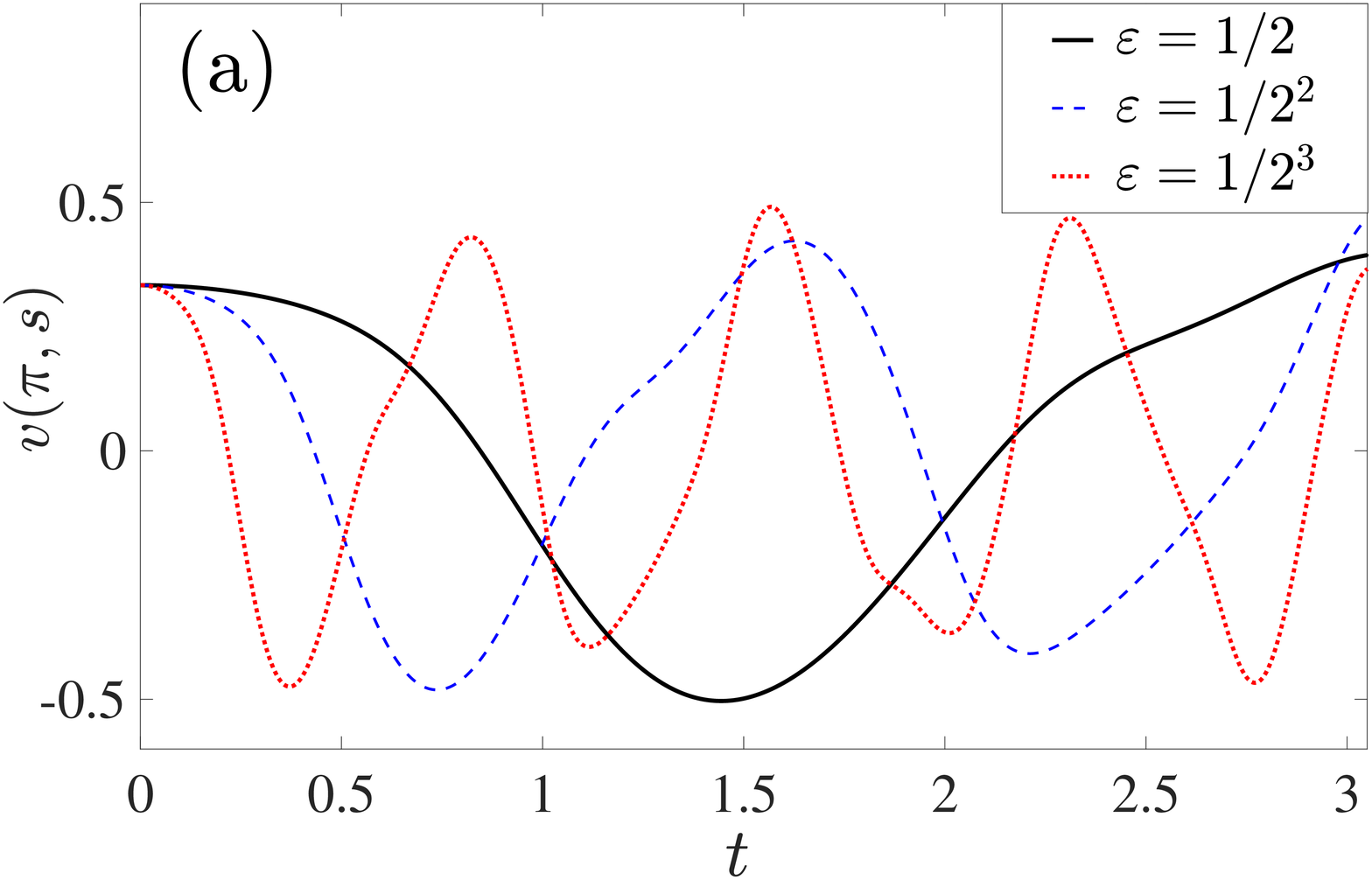}}
\end{minipage}
\begin{minipage}{0.48\textwidth}
\centerline{\includegraphics[width=6.8cm,height=4.5cm]{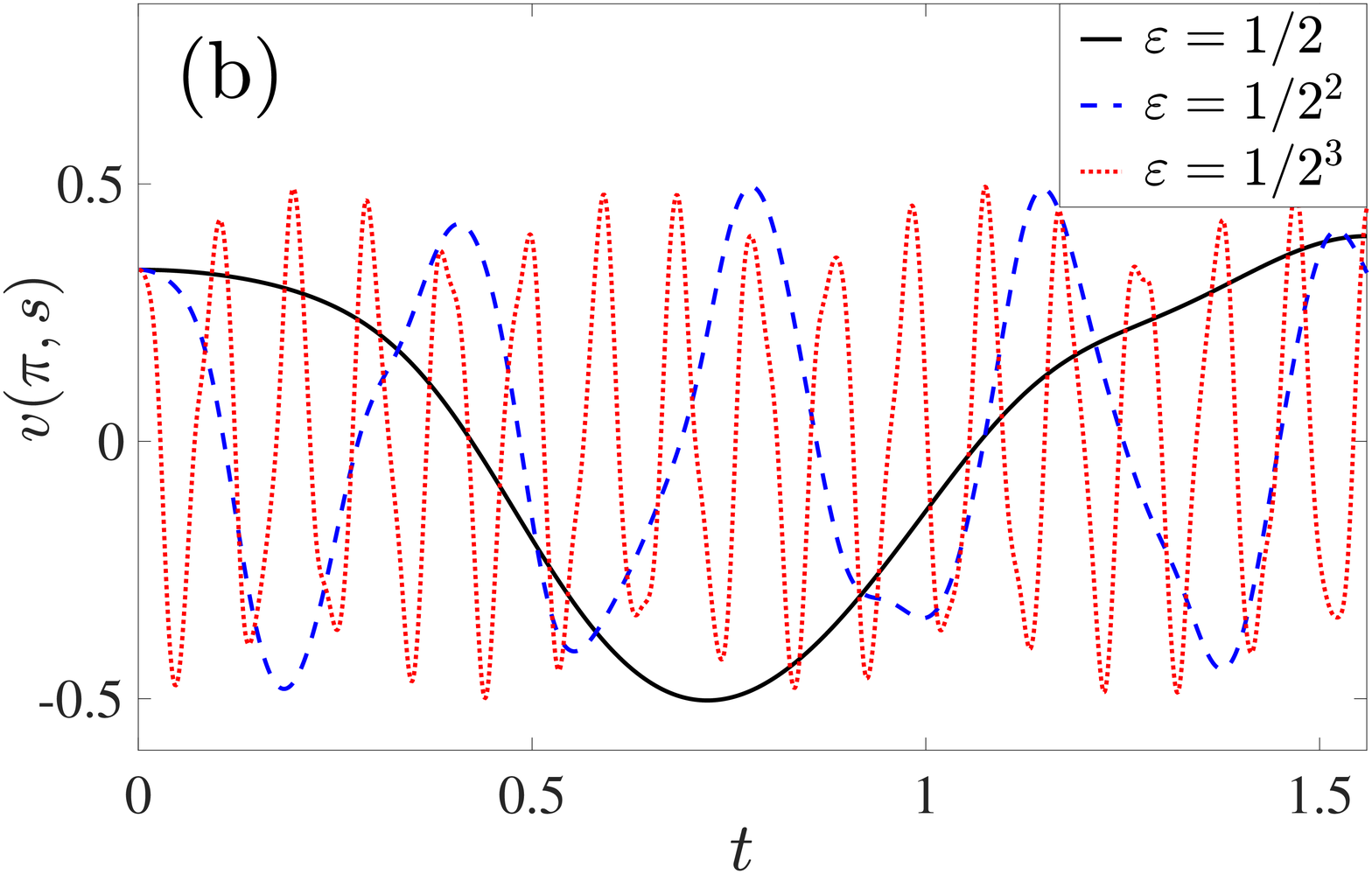}}
\end{minipage}
\vspace{-0.1in}
\caption{The solution $v(\pi, s)$ of the oscillatory NKGE (\ref{eq:HOE}) with $d=1$ and initial data \eqref{eq:long_initial} for different $\varepsilon$ and $\beta$: (a) $\beta = 1$, (b) $\beta = 2$.}
\label{fig:vt}
\end{figure}
\begin{figure}[ht!]
\begin{minipage}{0.48\textwidth}
\centerline{\includegraphics[width=6.8cm,height=4.5cm]{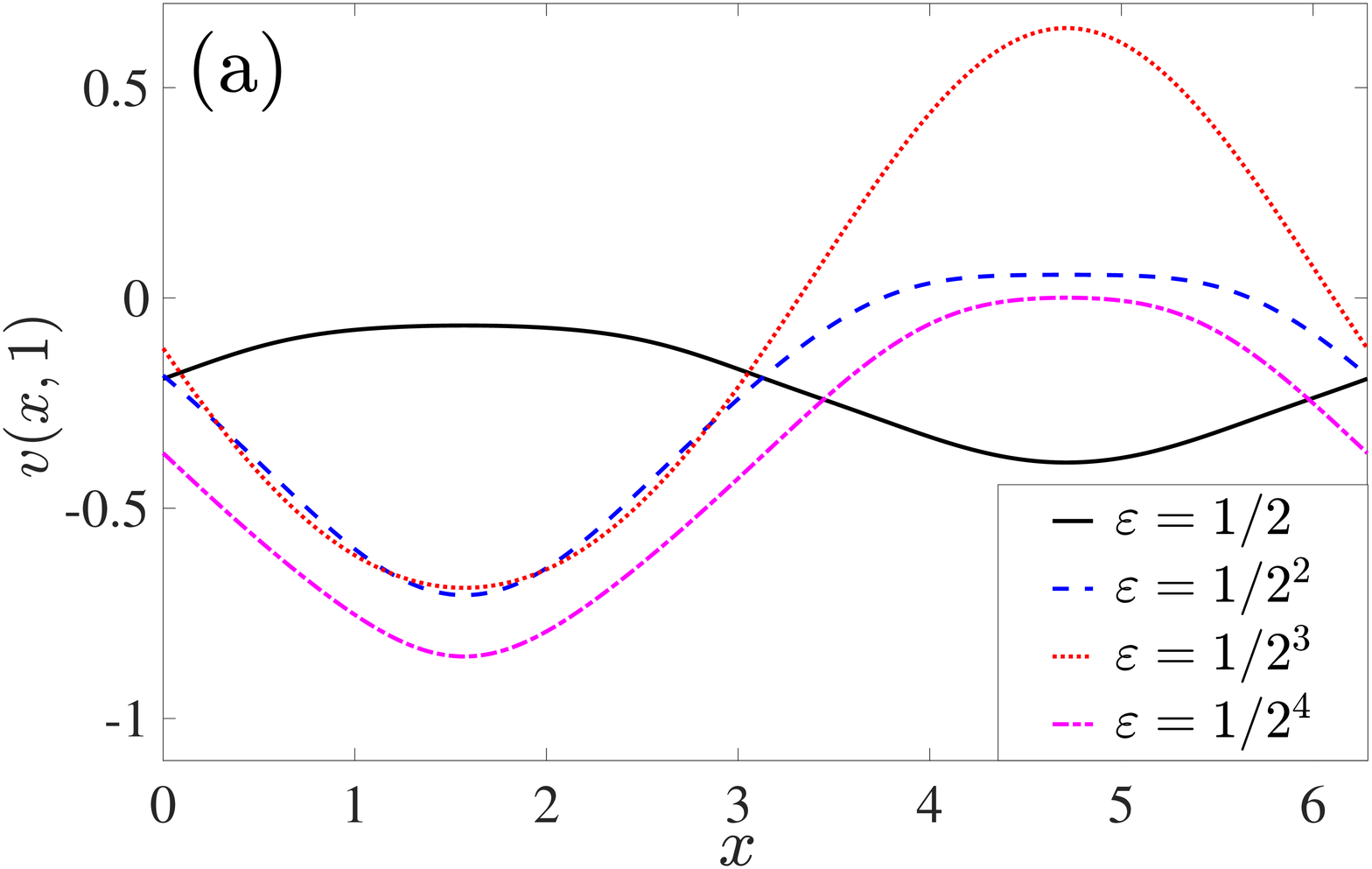}}
\end{minipage}
\begin{minipage}{0.48\textwidth}
\centerline{\includegraphics[width=6.8cm,height=4.5cm]{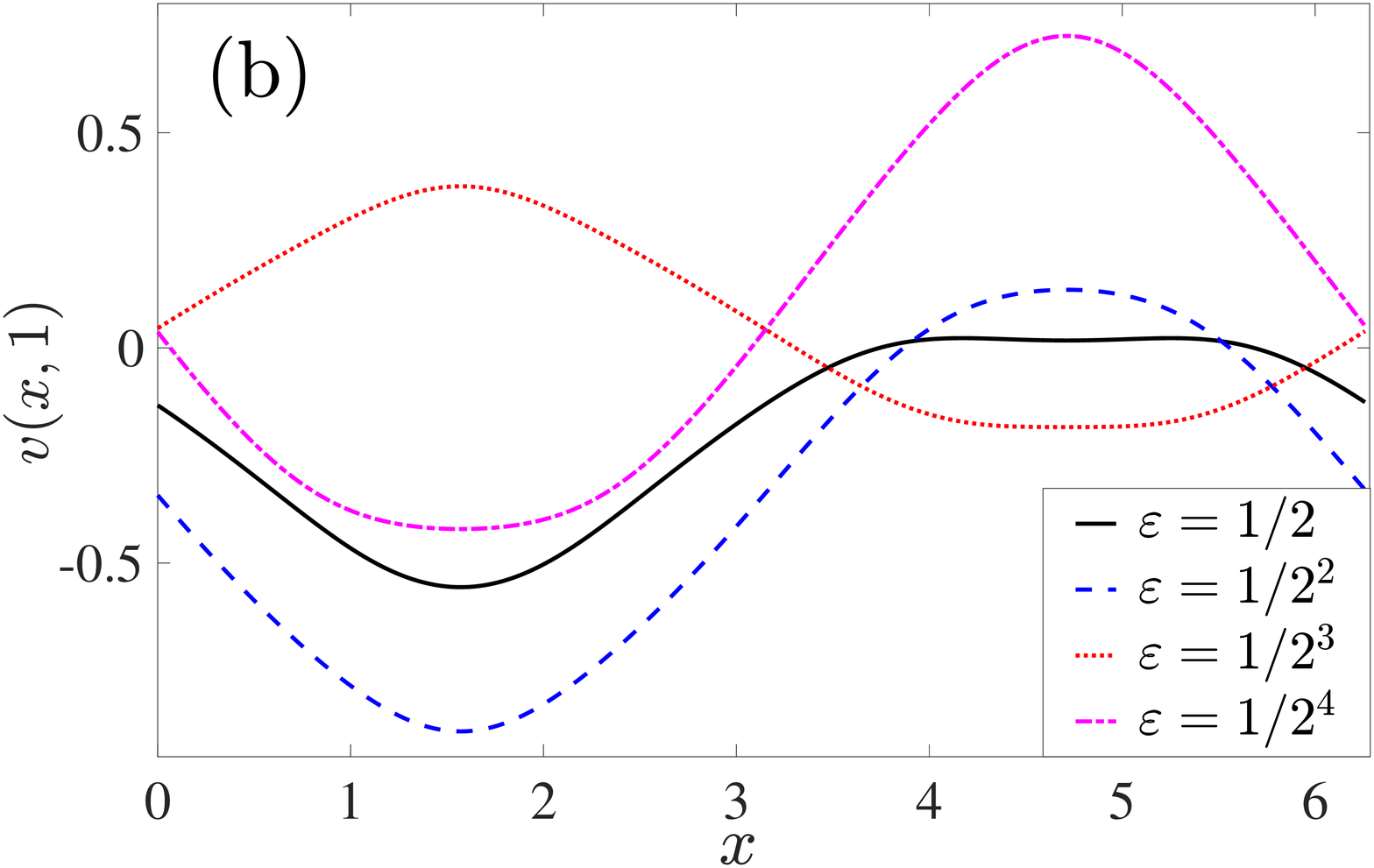}}
\end{minipage}
\vspace{-0.1in}
\caption{The solution $v(x, 1)$ of the oscillatory NKGE (\ref{eq:HOE}) with  $d=1$ and initial data \eqref{eq:long_initial} for different $\varepsilon$ and $\beta$: (a) $\beta = 1$, (b) $\beta = 2$.}
\label{fig:vx}
\end{figure}

The dynamics of the oscillatory NKGE \eqref{eq:HOE} up to the fixed time at $O(1)$ is equivalent to the long-time dynamics of the NKGE \eqref{eq:WNE} with weak nonlinearity up to the time at $O(\varepsilon^{-\beta})$. The analysis of the EWI-FS/EWI-FP method and the error estimates for the NKGE \eqref{eq:WNE} with weak nonlinearity in Sections 2$\&$3 can be extended for the oscillatory NKGE \eqref{eq:HOE}. Again, for simplicity of notations, the EWI-FS/EWI-FP method and the error estimates will be only presented in 1D and the results can be easily generalized to higher dimensions with minor modifications. In addition, the proofs for the error bounds are quite similar to those in Section 3 and we omit the details here for brevity. We adopt similar notations to those used in Sections 2\&3 except specified otherwise. In 1D, we consider the following oscillatory NKGE
\begin{equation}
\left\{
\begin{aligned}
&\varepsilon^{2\beta}\partial_{ss} v(x, s) - \partial_{xx} v(x, s) + v(x, s) + \varepsilon^{2} v^3(x, s) = 0,\ x \in \Omega = (a, b),\ s > 0,\\
&v(x, 0) = \phi(x), \quad \partial_s v(x, 0) = \varepsilon^{-\beta} \gamma(x) , \quad x \in \overline{\Omega} = [a, b].
\label{eq:HOE_1D}
\end{aligned}\right.
\end{equation}
%with periodic boundary conditions.

\subsection{EWI-FS/EWI-FP method}

Let $k = \Delta s > 0$ be the temporal step size and denote time steps as $s_n: = n k$ for $n \geq 0$. The Fourier spectral method for the oscillatory NKGE (\ref{eq:HOE_1D}) is to find $v_M(x, s) \in X_M$, i.e.,
\begin{equation}
v_M(x, s) = \sum_{l \in \mathcal{T}_M} \widehat{(v_M)}_l(s) e^{i\mu_l(x-a)}, \quad x \in \overline{\Omega}, \quad s \geq 0,
\label{eq:vm}
\end{equation}
such that
\begin{equation}
\varepsilon^{2\beta}\partial_{ss} v_M(x, s) - \partial_{xx} v_M(x, s) + v_M(x, s) + \varepsilon^2 P_M f(v_M(x, s) )= 0,\ x \in \overline{\Omega},\ s \geq 0,
\label{eq:51vm}
\end{equation}
with $f(v) = v^3$. The derivations of the EWI-FS/EWI-FP discretization for the oscillatory NKGE (\ref{eq:HOE_1D}) proceed in the analogous lines as those in Section 2 and we omit the details for brevity. Denote $\widehat{(v^n_M)}_l$ and $v^n_M(x)$ be the approximations of $\widehat{(v_M)}_l(s_n)$ and $v_M(x, s_n)$, respectively. Choosing $v^0_M(x) = (P_M \phi)(x)$, the Gautschi-type exponential wave integrator Fourier spectral (EWI-FS) discretization for the oscillatory NKGE  (\ref{eq:HOE_1D}) is
\begin{equation}
v^{n+1}_M(x) = \sum_{l \in \mathcal{T}_M} \widehat{(v^{n+1}_M)}_l e^{i\mu_l(x-a)},\quad x \in \overline{\Omega},\quad n\geq 0,
\label{eq:v1}
\end{equation}
where
\begin{equation}
\begin{split}
&\widehat{(v^1_M)}_l =\bar{p}_l\widehat{\phi}_l + \bar{q}_l \widehat{\gamma_l} + \bar{r}_l \widehat{(f(\phi))}_l, \quad l \in \mathcal{T}_M, \\
&\widehat{(v^{n+1}_M)}_l = -\widehat{(v^{n-1}_M)}_l + 2\bar{p}_l\widehat{(v^n_M)}_l + 2\bar{r}_l\widehat{(f(v^n_M))}_l,\quad {l \in \mathcal{T}_M},\quad n \geq 1,
\end{split}
\label{eq:EWI-FS_HOE}
\end{equation}
with $\bar{\zeta}_l = {\varepsilon^{-\beta}}\sqrt{1 + \mu^2_l }=O(\varepsilon^{-\beta})$ and the coefficients given as
\begin{equation}
\bar{p}_l= \cos(k\bar{\zeta}_l), \quad  \bar{q}_l = \frac{\sin(k\bar{\zeta}_l)}{\varepsilon^{\beta}\bar{\zeta}_l}, \quad \bar{r}_l =\frac{\varepsilon^2(\cos(k\bar{\zeta}_l)-1)}{(\varepsilon^{\beta}\bar{\zeta}_l)^2}.
\label{eq:barpqr}
\end{equation}

Similarly, let $v^n_j$ be the approximation of $v(x_j, s_n)$ and denote $v^0_j = \phi(x_j)$  $(j = 0, 1, \cdots, M)$, then we can obtain the following Gautschi-type exponential wave integrator Fourier pseudospectral (EWI-FP) discretization for the oscillatory NKGE (\ref{eq:HOE_1D}) as
\begin{equation}
v^{n+1}_j = \sum_{l \in \mathcal{T}_M} \widetilde{v}^{n+1}_l e^{i\mu_l(x_j-a)},\quad j = 0, 1, \cdots, M,\quad n\geq 0,
\label{eq:v_n1}
\end{equation}
where
\begin{equation}
\begin{split}
&\widetilde{v}^1_l = \bar{p}_l\widetilde{\phi}_l + \bar{q}_l\widetilde{\gamma_l} + \bar{r}_l\widetilde{(f(\phi))}_l, \quad l \in \mathcal{T}_M, \\
&\widetilde{v}^{n+1}_l= -\widetilde{v}^{n-1}_l + 2\bar{p}_l\widetilde{v}^n_l + 2\bar{r}_l\widetilde{(f(v^{n}))}_l,\quad{l \in \mathcal{T}_M},\quad n \geq 1,
\end{split}
\label{eq:EWI-FP_HOE}
\end{equation}
with the coefficients $\bar{p}_l$, $\bar{q}_l$ and $\bar{r}_l$ are given in \eqref{eq:barpqr}.

The EWI-FP \eqref{eq:v_n1}-\eqref{eq:EWI-FP_HOE} is also explicit, time symmetric and easy to extend to 2D and 3D. The memory cost is $O (M)$ and the computational cost per time step is $O (M {\rm log} M )$. Similar to Lemma \ref{thm:stability}, we have the following stability result for the EWI-FP \eqref{eq:v_n1}-\eqref{eq:EWI-FP_HOE} with the proof omitted here for brevity.

\begin{lemma} (stability)
Let $T_0 > 0$ be a fixed constant and denote
\begin{equation}
\bar{\sigma}_{\max}:=\max_{0 \leq n \leq T_0/k} \|v^n\|^2_{l^{\infty}}.
\end{equation}
The EWI-FP \eqref{eq:v_n1}-\eqref{eq:EWI-FP_HOE} is conditionally stable under the stability condition
\begin{equation}
 0 < k \leq \frac{2\varepsilon^{\beta}h}{\sqrt{\pi^2+h^2(1+\varepsilon^2\bar{\sigma}_{\max})}}, \quad h > 0,\quad 0 < \varepsilon \leq 1.	
\label{eq:stability_v} 
\end{equation}
\label{thm:stability_v}
\end{lemma}

\subsection{Main results}
In this section, we will establish the error estimates for the oscillatory NKGE \eqref{eq:HOE_1D}. Let $0 < T_0 < T^{\ast}$ with $T^{\ast}$ the maximum existence time of the solution. Again, motivated by the analytical results of the oscillatory NKGE \eqref{eq:HOE_1D} and the assumption ${\rm (A)}$, we make some assumptions on the exact solution $v(x, s)$ of the oscillatory NKGE \eqref{eq:HOE_1D}:

\begin{equation*}
\rm{(B)}\hspace{-1.5cm}
\begin{split}
 & v(x, s) \in  L^{\infty}\left([0, T_0]; H^{m_0}_p\right), \ \partial_s v(x, s) \in L^{\infty}\left([0, T_0]; W^{1,4}\right), \\
 & \partial_{ss} v(x, s) \in L^{\infty}\left([0, T_0]; H^1\right),\\
&\|v(x, s)\|_{L^{\infty}([0, T_0]; H^{m_0}_p)} \lesssim 1, \ \|\partial_s v(x, s)\|_{L^{\infty}([0, T_0]; W^{1,4})} \lesssim \frac{1}{\varepsilon^{\beta}},\\
&\|\partial_{ss} v(x, s)\|_{L^{\infty}([0, T_0]; H^1)} \lesssim  \frac{1}{\varepsilon^{2\beta}},\quad m_0 \geq 2.
\end{split}
\end{equation*}
Under the assumption (${B}$), let
\begin{equation}\nonumber
\begin{split}
& \overline{M_1} := \max_{\varepsilon \in (0, 1]}\left\{ \|v(x, s)\|_{L^{\infty}([0, T_0]; L^{\infty})} + \varepsilon^{\beta}\|\partial_s v(x, s)\|_{L^{\infty}([0, T_0]; L^{\infty})}\right\} \lesssim 1,\\
&\overline{M_2} := \sup_{v \neq 0, |v| \leq 1 + \overline{M_1}} |v|^2 \lesssim 1.
\end{split}
\end{equation}
Assuming
\begin{equation}
k \leq  \min \left\{\frac{1}{4}\varepsilon^{\beta}, \frac{\varepsilon^{\beta}\pi h}{3\sqrt{ h^2+\pi^2 + \varepsilon^2 \overline{M_2} h^2}}\right\},\quad 0 < \varepsilon \leq 1, \quad 0 \leq \beta \leq 2,
\label{eq:con_k}
\end{equation}
taking $\tau=k\varepsilon^{-\beta}$ and noticing the error bounds in Theorem \ref{thm:EWI-FS}, we can immediately obtain the error bounds of the  EWI-FS \eqref{eq:v1}-\eqref{eq:EWI-FS_HOE} (The results for the EWI-FP \eqref{eq:v_n1}-\eqref{eq:EWI-FP_HOE} are quite similar and the details are skipped here for brevity):
\begin{theorem}
Let $v^n_M(x)$ be the approximation obtained from the EWI-FS \eqref{eq:v1}-\eqref{eq:EWI-FS_HOE}, under the stability condition \eqref{eq:stability_v} and the assumptions (B) and (\ref{eq:con_k}), there exist constants $h_0> 0$ and $k_0 > 0$ sufficiently small and independent of $\varepsilon$, such that for any $0 < \varepsilon \leq 1$ and $0 \leq \beta \leq 2$, when $0 < h \leq h_0$, $0 < k \leq \varepsilon^{\alpha^{\ast}}k_0$ with $\alpha^{\ast} = \max\{0, 3\beta/2-1\}$, we have
\begin{equation}
\begin{split}
&\|v(x, s_n) - v^n_M(x)\|_{\lambda} \lesssim h^{m_0-\lambda} + \varepsilon^{2-3\beta}{k^2},\quad \lambda=0,1,\\
& \|v^n_M(x)\|_{L^{\infty}} \leq 1 + \overline{M_1}, \quad 0 \leq n \leq \frac{T_0}{k}.
\end{split}\label{eq:error_HOE}
\end{equation}
\label{thm:EWI-FS_HOE}
\end{theorem}

Based on Theorem \ref{thm:EWI-FS_HOE}, for a given accuracy bound $\delta_0 > 0$, the $\varepsilon$-scalability of the EWI-FS/EWI-FP method for the oscillatory NKGE \eqref{eq:HOE_1D} is:
\begin{equation*}
h = O(1),\ k = O(\varepsilon^{\alpha^{\ast}}\sqrt{\delta_0}) = O(\varepsilon^{\alpha^{\ast}}),\ \alpha^{\ast} = \max\{0, 3\beta/2-1\}, \ 0\leq \beta \leq 2.
\end{equation*}
This indicates that, in order to obtain ``correct" numerical solution in the fixed time interval, one has to take the meshing strategy: $h=O(1)$ and $k=O(1)$, when $0\leq\beta\leq 2/3$; and resp., $h=O(1)$ and $k=O(\varepsilon^{3\beta/2-1})$, when $2/3<\beta\leq 2$. These results are useful for choosing mesh size and time step in practical computations such that the numerical results are trustable.

\subsection{Numerical results of the oscillatory NKGE in the whole space}
To avoid repetitions, we consider the following oscillatory NKGE in $d$-dimensional ($d = 1, 2, 3$) whole space
\begin{equation}
\left\{
\begin{aligned}
&\varepsilon^{2\beta }\partial_{ss} v({\bf{x}}, s) - \Delta v({\bf{x}}, s) + v({\bf{x}}, s) + \varepsilon^{2} v^3({\bf{x}}, s) = 0,\quad {\bf{x}} \in \mathbb{R}^d,\quad s > 0,\\
&v({\bf{x}}, 0) = \phi({\bf{x}}), \quad \partial_s v({\bf{x}}, 0) = \varepsilon^{-\beta}\gamma({\bf{x}}), \quad {\bf{x}} \in \mathbb{R}^d .
\end{aligned}\right.
\label{eq:HOE_whole}
\end{equation}

The solution of the oscillatory NKGE \eqref{eq:HOE_whole} propagates waves with wavelength at $O(1)$ in space and $O(\varepsilon^{\beta})$ in time, and wave speed in space at $O(\varepsilon^{-\beta})$. To illustrate the outgoing waves of the solution to the oscillatory NKGE \eqref{eq:HOE_whole} further, Figure \ref{fig:HOE_x} shows the solutions $v(x, 1)$ of the oscillatory NKGE \eqref{eq:HOE_whole} with $d=1$ and the initial data
\begin{equation}
\phi(x) = 1/(e^{x^2}+e^{-x^2}) \quad \mbox{and} \quad \gamma(x) = 2e^{-x^2},\qquad x \in \mathbb{R}.
\label{eq:initial_1}
\end{equation}
\begin{figure}[htb]
\begin{minipage}{\textwidth}
\centerline{\includegraphics[width = 14cm, height = 6cm]{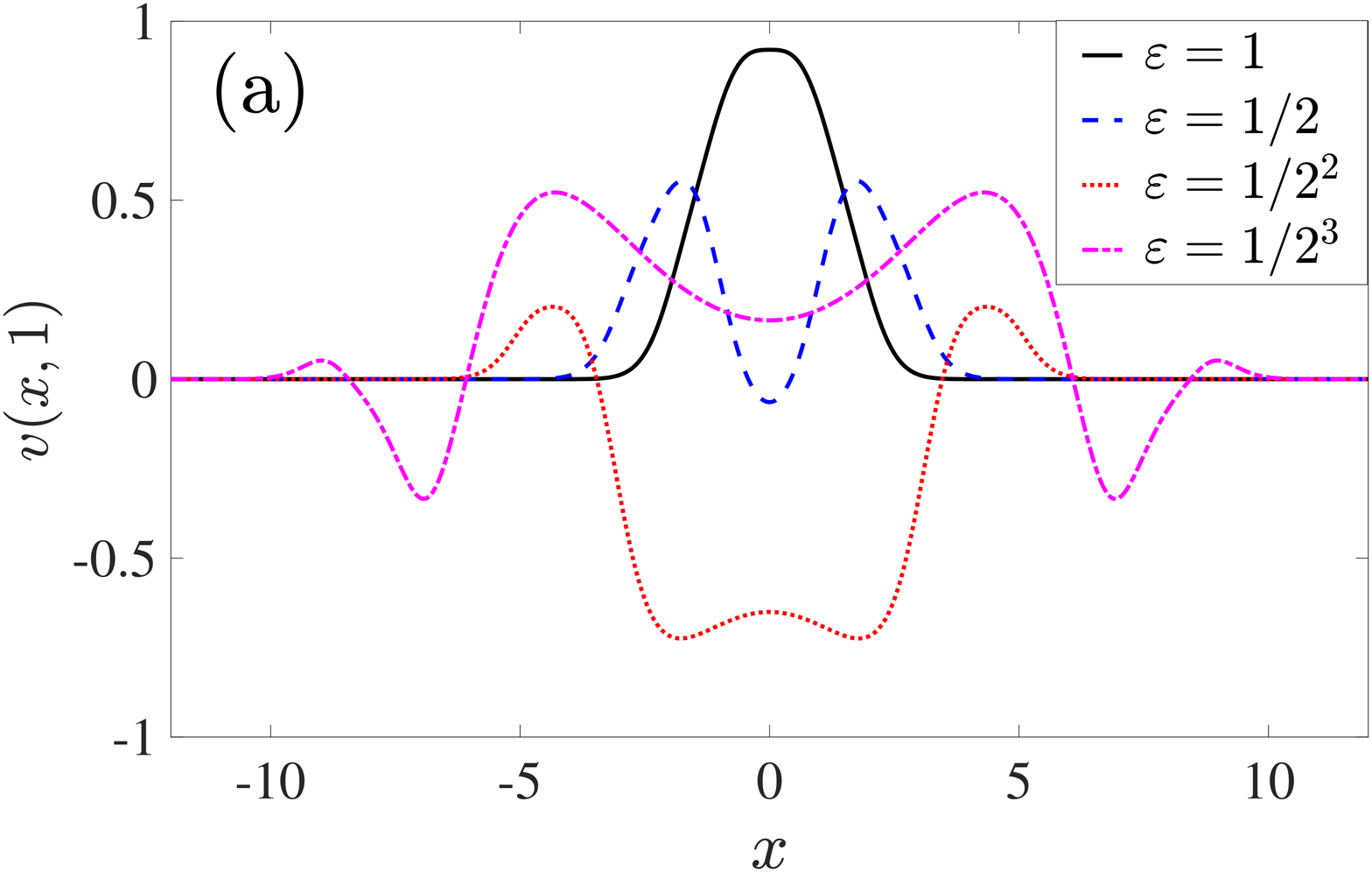}}
\end{minipage}
\hfill
\begin{minipage}{\textwidth}
\centerline{\includegraphics[width = 14cm, height = 6cm]{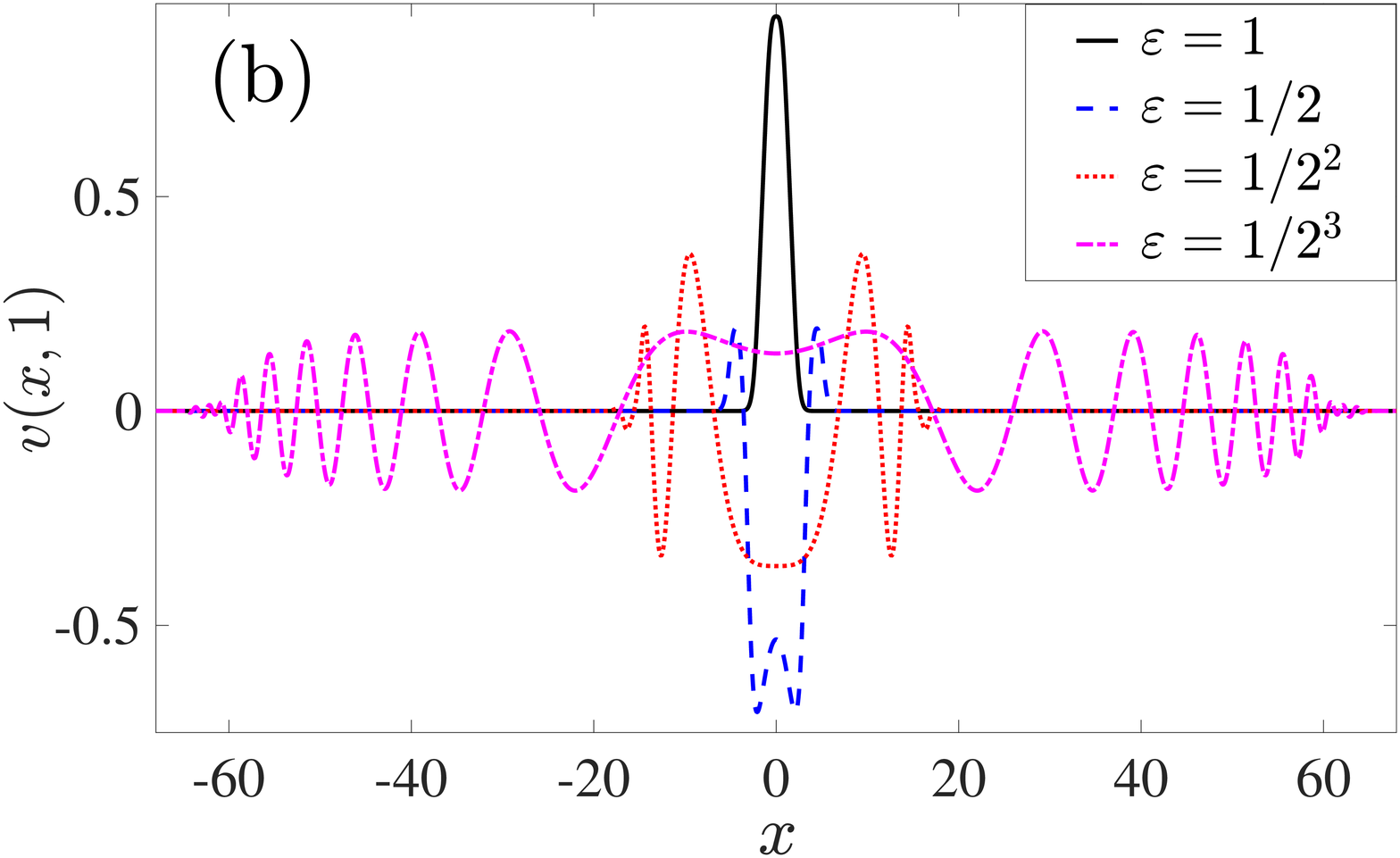}}
\end{minipage}
\caption{The solutions $v(x, 1)$ of the oscillatory NKGE \eqref{eq:HOE_whole} for different $\varepsilon$ and $\beta$: (a) $\beta = 1$, (b) $\beta = 2$.}
\label{fig:HOE_x}
\end{figure}

Similar to those in the literature, due to the fast decay of the solution of the oscillatory NKGE \eqref{eq:HOE_whole}  at the far field (see \cite{BD, DB, SV} and references therein), in practical computations, we usually truncate the original whole space problem onto a bounded domain with periodic boundary conditions, which is large enough such that the truncation error is negligible. Due to the rapid outgoing waves with wave speed $O(\varepsilon^{-\beta})$, the computational domain $\Omega_{\varepsilon}$ has to be chosen as $\varepsilon$-dependent.

In the following, we report numerical results of the oscillatory NKGE \eqref{eq:HOE_whole} with $d = 1$ and the initial data \eqref{eq:initial_1}. The problem is solved on a bounded domain $\Omega_{\varepsilon} = [-4 - \varepsilon^{-\beta}, 4 + \varepsilon^{-\beta}]$. The ``exact" solution is obtained numerically by the time-splitting Fourier pseudospectral method with a very fine mesh size and a very small time step, e.g. $h_e = 1/16$ and $k_e = 10^{-5}$. Denote $v^n_{h, k}$ as the numerical solution at $s = s_n$ by the EWI-FP \eqref{eq:v_n1}-\eqref{eq:EWI-FP_HOE} with mesh size $h$ and time step $k$. The errors are denoted as $\tilde{e}(x, s_n) \in X_M$ with $\tilde{e}(x, s_n) = v(x, s_n) - I_M(v^n_{h, k})(x)$. We also measure the $H^1$ norm of $\tilde{e}(x, s_n)$ and the errors are displayed at $s = 1$ with different $\varepsilon$ and $\beta$. For the oscillatory NKGE, we study the following three cases with respect to different $\beta$:

Case I. Classical case, i.e., $\beta = 0$;

Case II. Intermediately oscillatory case, i.e., $\beta = 1$;

Case III. Highly oscillatory case, i.e., $\beta = 2$.

For spatial error analysis, we fix the time step as $k = 10^{-5}$ such that the temporal error can be ignored, and solve the oscillatory NKGE \eqref{eq:HOE_whole} with different mesh size $h$. Table \ref{tab:HOE_h} depicts the spatial errors for $\beta = 0$, $\beta = 1$ and $\beta = 2$. For temporal error analysis, a very fine mesh size $h = 1/16$ is chosen such that the spatial errors can be neglected. Tables \ref{tab:HOE_beta0_t}-\ref{tab:HOE_beta2_t} show the temporal errors for $\beta = 0$, $\beta = 1$ and $\beta = 2$, respectively.

\begin{table}
\caption{Spatial errors of the EWI-FP \eqref{eq:v_n1}-\eqref{eq:EWI-FP_HOE} for the oscillatory NKGE \eqref{eq:HOE_whole} with \eqref{eq:initial_1} for different $\beta$ and $\varepsilon$.}
\centering
\begin{tabular}{c|ccccc}
\hline
&$\|\tilde{e}(\cdot, 1)\|_1$ & \;\;\; $h_0 = 1 $ \;\;\; & \;\;\; $h_0/2 $ \;\;\; & \;\;\; $h_0/2^2$ \;\;\; & \;\;\;  $h_0/2^3$ \;\;\; \\
\hline
\multirow{5}{*}{$\beta=0$}
& $\varepsilon_0 = 1$ & 3.66E-2 & 1.15E-3 & 7.13E-6 & 3.34E-7 \\
&$\varepsilon_0 / 2$ & 5.15E-2 & 5.43E-4 & 2.56E-6 & 3.26E-7 \\
&$\varepsilon_0 / 2^2$ & 5.61E-2 & 6.35E-4 & 1.64E-6 & 3.10E-7 \\
&$\varepsilon_0 / 2^3$ & 5.73E-2 & 6.89E-4 & 1.56E-6 & 2.96E-7\\
&$\varepsilon_0 / 2^4$ & 5.76E-2 & 7.04E-4 & 1.56E-6 & 3.06E-7 \\
\hline
\multirow{5}{*}{$\beta=1$}
& $\varepsilon_0 = 1$ & 3.66E-2 & 1.15E-3 & 7.13E-6 & 3.34E-7 \\
&$\varepsilon_0 / 2$  &1.08E-1 & 1.23E-3 & 7.88E-6 & 8.99E-7 \\
&$\varepsilon_0 / 2^2$ & 1.78E-1 & 4.00E-3 & 1.23E-5 & 6.68E-7 \\
&$\varepsilon_0 / 2^3$ & 2.26E-1 & 9.90E-3 & 2.72E-5 & 1.20E-6 \\
&$\varepsilon_0 / 2^4$ & 4.43E-2 & 1.81E-2 & 5.90E-5 & 3.71E-7 \\
\hline
\multirow{5}{*}{$\beta=2$}
& $\varepsilon_0 = 1$ & 3.66E-2 & 1.15E-3 & 7.13E-6 & 3.34E-7  \\
&$\varepsilon_0 / 2$ & 1.64E-1 & 3.43E-3 & 1.72E-5 & 5.70E-7 \\
&$\varepsilon_0 / 2^2$ & 4.94E-2 & 1.78E-2 & 6.16E-5 & 3.66E-7 \\
&$\varepsilon_0 / 2^3$ & 2.73E-1 & 1.83E-2 & 6.03E-5 & 9.60E-8 \\
&$\varepsilon_0 / 2^4$ & 1.60E-1 & 1.90E-2 &  8.86E-5 & 2.75E-7 \\
\hline
\end{tabular}
\label{tab:HOE_h}
\end{table}

\begin{table}
\caption{Temporal errors of the EWI-FP \eqref{eq:v_n1}-\eqref{eq:EWI-FP_HOE} for the oscillatory NKGE (\ref{eq:HOE_whole}) with \eqref{eq:initial_1} and $\beta=0$.}
\centering
\begin{tabular}{ccccccc}
\hline
$\|\tilde{e}(\cdot, 1)\|_1$ & $k_0 = 0.1 $ & $k_0/2 $ &$k_0/2^2 $  &  $k_0/2^3 $  &  $k_0/2^4$  &  $k_0/2^5$  \\
\hline
$\varepsilon_0 = 1$ & 1.08E-2 & 2.68E-3 & 6.70E-4 & 1.67E-4 & 4.18E-5 & 1.05E-5  \\
order & - & 2.01 & 2.00 & 2.00 & 2.00 & 1.99 \\
\hline
$\varepsilon_0/2$ & 3.99E-3 & 9.95E-4 & 2.48E-4 & 6.21E-5 & 1.55E-5 & 3.88E-6  \\
order & - & 2.00 & 2.00 & 2.00 & 2.00 & 2.00 \\
\hline
$\varepsilon_0/2^2$ & 1.15E-3 & 2.86E-4 & 7.15E-5 & 1.79E-5 & 4.47E-6 & 1.12E-6  \\
order & - & 2.01 & 2.00 & 2.00 & 2.00 & 2.00 \\
\hline
$\varepsilon_0/2^3$ & 2.98E-4 & 7.43E-5 & 1.86E-5 & 4.64E-6 & 1.16E-6 & 2.90E-7  \\
order & - & 2.00 & 2.00 & 2.00 & 2.00 & 2.00 \\
\hline
$\varepsilon_0/2^4$ & 7.52E-5 & 1.88E-5 & 4.69E-6 & 1.17E-6 & 2.93E-7 & 7.32E-8 \\
order & - & 2.00 & 2.00 & 2.00 & 2.00 & 2.00 \\
\hline
\end{tabular}
\label{tab:HOE_beta0_t}
\end{table}

\begin{table}
\caption{Temporal errors of the EWI-FP \eqref{eq:v_n1}-\eqref{eq:EWI-FP_HOE} for the oscillatory NKGE \eqref{eq:HOE_whole} with \eqref{eq:initial_1} and $\beta=1$.}
\centering
\begin{tabular}{ccccccc}
\hline
$\|\tilde{e}(\cdot, 1)\|_1$ &$k_0 = 0.1 $ & $k_0/2 $ &$k_0/2^2 $ & $k_0/2^3 $ & $k_0/2^4$  & $k_0/2^5$ \\
\hline
$\varepsilon_0 = 1$ & \bf{1.08E-2} & 2.68E-3 & 6.70E-4 & 1.67E-4 & 4.18E-5 & 1.05E-5 \\
order & \bf{-} & 2.01 & 2.00 & 2.00 & 2.00 & 1.99 \\
\hline
$\varepsilon_0 / 2 $  & 2.57E-2 & \bf{6.26E-3} & 1.55E-3 & 3.88E-4 & 9.70E-5 & 2.42E-5  \\
order & -  & \bf{2.04} & 2.01 & 2.00 & 2.00 & 2.00  \\
\hline
$\varepsilon_0 / 2^2 $ & 5.01E-2 & 1.15E-2 & \bf{2.81E-3} & 7.00E-4 & 1.75E-4 & 4.35E-5 \\
order & -  & 2.12 & \bf{2.03} & 2.01 & 2.00 & 2.00  \\
\hline
$\varepsilon_0 / 2^3 $ & 2.57E-1 & 2.12E-2 & 4.78E-3 &  \bf{1.17E-3} & 2.91E-4 & 7.28E-5  \\
order & -  & 3.60 & 2.15 & \bf{2.03} & 2.01 & 2.00 \\
\hline
$\varepsilon_0 / 2^4$ & 1.70E-1 & 1.09E-1 & 7.47E-3 & 1.70E-3 & \bf{4.18E-4} & 1.04E-4  \\
order & - & 0.64 & 3.87 & 2.14 & \bf{2.02} & 2.01  \\
\hline
\end{tabular}
\label{tab:HOE_beta1_t}
\end{table}

\begin{table}
\caption{Temporal errors of the EWI-FP \eqref{eq:v_n1}-\eqref{eq:EWI-FP_HOE} for the oscillatory NKGE \eqref{eq:HOE_whole} with \eqref{eq:initial_1} and $\beta=2$.}
\centering
\begin{tabular}{ccccccc}
\hline
$\|\tilde{e}(\cdot, 1)\|_1$ &$k_0 = 0.1 $ & $k_0/4 $ &$k_0/4^2 $ & $k_0/4^3 $ & $k_0/4^4$  & $k_0/4^5$ \\
\hline
$\varepsilon_0 = 1$ & \bf{1.08E-2} & 6.70E-4 & 4.18E-5 & 2.62E-6 & 1.64E-7 & $<$1E-7 \\
order & \bf{-} & 2.01 & 2.00  & 2.00 & 2.00 & - \\
\hline
$\varepsilon_0 / 2 $  & 1.98E-1 & \bf{1.10E-2} & 6.86E-4 & 4.28E-5 & 2.68E-6 & 1.64E-7  \\
order & -  & \bf{2.08} & 2.00 & 2.00 & 2.00 & 2.02  \\
\hline
$\varepsilon_0 / 2^2 $ & 3.25E+0 & 1.22E-1 & \bf{6.82E-3} & 4.24E-4 & 2.65E-5 & 1.65E-6 \\
order & -  & 2.37 & \bf{2.08} & 2.00 & 2.00 & 2.00  \\
\hline
$\varepsilon_0 / 2^3 $ & 1.33E+0 & 1.95E+1 & 4.71E-2 &  \bf{2.52E-3} & 1.57E-4 & 9.81E-6  \\
order & -  & -0.28 & 2.69 & \bf{2.11} & 2.00 & 2.00 \\
\hline
$\varepsilon_0 / 2^4$ & 4.81E-1 & 5.33E-1 & 9.88E-1 & 1.69E-2 & \bf{7.94E-4} & 4.95E-5  \\
order & - & -0.07 & -0.45 & 2.93 & \bf{2.21} & 2.00  \\
\hline
\end{tabular}
\label{tab:HOE_beta2_t}
\end{table}

From Tables \ref{tab:HOE_h}-\ref{tab:HOE_beta2_t} and additional numerical results not shown here for brevity, we can draw the following observations:

(i) In space, the EWI-FP \eqref{eq:v_n1}-\eqref{eq:EWI-FP_HOE} converges uniformly with exponential convergence rate for any fixed $0 < \varepsilon \leq 1$ and $0 \leq \beta \leq 2$ (cf. each row in Table \ref{tab:HOE_h}).% and the spatial errors are almost independent of $\varepsilon$ (cf. each column in Table \ref{tab:HOE_h}).

(ii) In time, for any fixed $\varepsilon = \varepsilon_0>0$, the EWI-FP \eqref{eq:v_n1}-\eqref{eq:EWI-FP_HOE} is uniformly second-order accurate (cf. the first rows in Tables \ref{tab:HOE_beta0_t}-\ref{tab:HOE_beta2_t}), which agree with the results in the literature. For the classical case, i.e., $\beta = 0$, Table \ref{tab:HOE_beta0_t}  indicates that the temporal error of the EWI-FP \eqref{eq:v_n1}-\eqref{eq:EWI-FP_HOE} behaves like $O(\varepsilon^2k^2)$ (cf. each row and column in Table \ref{tab:HOE_beta0_t}); When $\beta = 1$, the EWI-FP\eqref{eq:v_n1}-\eqref{eq:EWI-FP_HOE} converges quadratically in time when $k \lesssim \varepsilon$ (cf. each row in the upper triangle above the diagonal (corresponding to $k \sim \varepsilon$ and being labelled in bold letters) in Table \ref{tab:HOE_beta1_t}). When $\beta = 2$, the EWI-FP\eqref{eq:v_n1}-\eqref{eq:EWI-FP_HOE} converges quadratically in time when $k \lesssim \varepsilon^2$ (cf. each row in the upper triangle above the diagonal (corresponding to $k \sim \varepsilon^2$ and being labelled in bold letters) in Table \ref{tab:HOE_beta2_t}). In summary, our numerical results confirm our rigorous error estimates.

%%%%%%%%%%%%%%%%%%%
% %    Section 6  Conclusion
%%%%%%%%%%%%%%%%%%%
\section{Conclusion}
The Gautschi-type exponential wave integrator Fourier spectral/ pseudospectral (EWI-FS/EWI-FP) method was proposed and analyzed for the long-time dynamics of the nonlinear Klein-Gordon equation (NKGE) with a weak cubic nonlinearity. Uniform error bounds of the EWI-FS/EWI-FP discretization were rigorously established up to the time at $O(\varepsilon^{-\beta})$ with $ 0 \leq \beta \leq 2$ and $\varepsilon \in (0,1]$ a dimensionless parameter used to characterize the strength of the nonlinearity. In addition, the EWI-FS/EWI-FP method was also applied to solve an oscillatory NKGE with the error bounds obtained straightforwardly. Based on the error bounds, in order to obtain ``correct'' numerical approximations of the oscillatory NKGE, the $\varepsilon-$scalability (or meshing strategy requirement) of the EWI-FS/EWI-FP discretization has to be taken as: $h = O(1), k = O(\varepsilon^{\alpha^{\ast}})$ with $\alpha^{\ast} = \max\{0, 3\beta/2-1\}$. Finally, extensive numerical results were reported to confirm our error estimates and demonstrate that they are optimal and sharp.

\section*{Acknowledgements}
The authors would like to thank Professor Weizhu Bao for his valuable suggestions and comments. This work was partially supported by the Ministry of Education of Singapore grant R-146-000-290-114 (Y. Feng) and the Fundamental Research Funds for the Central Universities 531107051208 and the NSFC grant 11901185 (W. Yi). Part of the work was done when the authors were visiting the Department of Mathematics and the Institute for Mathematical Sciences at the National University of Singapore in 2018 and 2020.

\bibliographystyle{siamplain}

\begin{thebibliography}{10}
	
	
	%B
	\bibitem{BC}
	{\sc  W. Bao and Y. Cai},
	Optimal error estimates of finite difference methods for the Gross-Pitaevskii equation with angular momentum rotation,
	Math. Comp., 82 (2013), pp. 99-128.
	
	\bibitem{BC1}
	{\sc  W. Bao and Y. Cai},
	Uniform and optimal error estimates of an exponential wave integrator sine pseudospectral method for the nonlinear Sch{\"o}dinger equation with wave operator,
	SIAM J. Numer. Anal., 52 (2014), pp. 1103-1127.
	
	 \bibitem{BCJY}
    {\sc W. Bao, Y. Cai, X. Jia and J. Yin},
     Error estimates of numerical methods for the nonlinear Dirac equation in the nonrelativistic limit regime, Sci. China Math., 59 (2016), pp. 1461-1494.
	
	\bibitem{BCZ}
	{\sc W. Bao, Y. Cai and X. Zhao},
	A uniformly accurate multiscale time integrator pseudospectral method for the Klein-Gordon equation in the nonrelativistic limit regime,
	SIAM J. Numer. Anal., 52 (2014), pp. 2488-2511.
	
	\bibitem{BD}
	{\sc  W. Bao and X. Dong},
	Analysis and comparison of numerical methods for the Klein-Gordon equation in the nonrelativistic limit regime,
	Numer. Math., 120 (2012), pp. 189-229.
			
	\bibitem{BDX}
	{\sc  W. Bao, X. Dong and J. Xin},
	Comparisons between sine-Gordon equation and perturbed nonlinear Schr{\"o}dinger equations for modeling light bullets beyond critical collapse,
	Phys. D, 239 (2010), pp. 1120-1134.	

        \bibitem{BDZ}
       {\sc W. Bao, X. Dong and X. Zhao},
       An exponential wave integrator pseudospectral method for the Klein-Gordon-Zakharov system, SIAM J. Sci. Comput., 35 (2013), pp. A2903-A2927.

       \bibitem{BDZ0}
       {\sc W. Bao, X. Dong and X. Zhao},
       Uniformly accurate multiscale time integrators for highly oscillatory second order differential equations, J. Math. Study, 47 (2014), pp. 111-150.

         \bibitem{BFSU}
        {\sc W. Bao, Y. Feng and C. Su},
        Uniform error bounds of a time-splitting spectral method for the long-time dynamics of the nonlinear Klein-Gordon equation with weak nonlinearity, arXiv: 2001.10868.
        
        \bibitem{BFY}
        {\sc W. Bao, Y. Feng and W. Yi},
        Long time error analysis of finite difference time domain methods for the nonlinear Klein-Gordon equation with weak nonlinearity, Commun. Comput. Phys., 26 (2019), pp.  1307-1334.
        
        \bibitem{BY}
        {\sc W. Bao and L. Yang},
        Efficient and accurate numerical methods for the Klein-Gordon-Schr\"{o}dinger equations, J. Comput. Phys., 225 (2007), pp. 1863-1893.


        \bibitem{BFS}
        {\sc S. Baumstark, E. Faou and K. Schratz},
        Uniformly accurate exponential-type integrators for Klein-Gordon equations with asymptotic convergence to the classical NLS splitting, Math. Comp., 87 (2018), pp.   1227-1254.

        \bibitem{BJ}
	{\sc J. Bourgain},
	Construction of approximative and almost periodic solutions of perturbed linear Schr{\"o}dinger and wave equations,
	Geom. Funct. Anal., 6 (1996), pp. 201-230.
	
	 \bibitem{BV}
	{\sc  P. Brenner and W. von Wahl},
	Global classical solutions of nonlinear wave equations,
	Math. Z., 176 (1981), pp. 87-121.
	
	%C
	\bibitem{CCLM}
	{\sc P. Chartier, N. Crouseilles, M. Lemou and F. M\'ehats},
	Uniformly accurate numerical schemes for highly oscillatory Klein-Gordon and nonlinear Schr\"odinger equations,
	Numer. Math., 129 (2015), pp. 211-250.
	
	\bibitem{CCLMV}
	{\sc P. Chartier, N. Crouseilles, M. Lemou, F. M\'ehats and G. Vilmart}, A New Class of Uniformly Accurate Numerical Schemes for Highly Oscillatory Evolution Equations. Found. Comput. Math. 20(2020), pp. 1-33.
	
	\bibitem{CE}
	{\sc S. C. Chikwendu and C. V. Easwaran},
	Multiple-scale solution of initial-boundary value problems for weakly nonlinear Klein-Gordon equations on the semi-infinite line,
	SIAM J. Appl. Math., 52 (1992), pp. 946-958.
	
	\bibitem{CHL}
	{\sc D. Cohen, E. Hairer and Ch. Lubich},
	Conservation of energy, momentum and actions in numerical discretizations of non-linear wave equations, Numer. Math., 110 (2008), pp. 113-143.
	
	\bibitem{CHL1}
	{\sc D. Cohen, E. Hairer and Ch. Lubich},
	Long-time analysis of nonlinearly perturbed wave equations via modulated Fourier expansions,
	Arch. Ration. Mech. Anal., 187 (2008), pp. 341-368.
	
	%D
	
	\bibitem{D}
	{\sc J.-M. Delort},
	{Temps d'existence pour l'{\'e}quation de Klein-Gordon semi-lin{\'e}aire {\`a} donn{\'e}es petites p{\'e}riodiques},
	Amer. J. Math., 120 (1998), pp. 663-689.
	
	\bibitem{D2}
	{\sc J.-M. Delort},
	{On long time existence for small solutions of semi-linear Klein-Gordon equations on the torus},
	J. Anal. Math., 107 (2009), pp. 161-194.
	
	
	\bibitem{DEGM}
	{\sc R. K. Dodd, J. C. Eilbeck, J. D. Gibbon and H. C. Morris},
	Solitons and Nonlinear Wave Equations,
	Academic Press, New York, 1984.	
	
	\bibitem{DXZ}
	{\sc  X. Dong, Z. Xu and X. Zhao},
	On time-splitting pseudospectral discretization for nonlinear Klein-Gordon equation in nonrelativistic limit regime,
	Commun. Comput. Phys., 16 (2014), pp. 440-466.
	
	\bibitem{DB}
	{\sc D. B. Duncan},
	Sympletic finite difference approximations of the nonlinear Klein-Gordon equation,
	SIAM J. Numer. Anal., 34 (1997), pp. 1742-1760.
	
	%F
	
	\bibitem{FS}
	{\sc E. Faou and K. Schratz},
	Asymptotic preserving schemes for the Klein-Gordon equation in the nonrelativistic limit regime,
	Numer. Math., 126 (2014), pp. 441-469.
	
	\bibitem{FZ}
	{\sc D. Fang and Q. Zhang},
	Long-time existence for semi-linear Klein--Gordon equations on tori,
	J. Differential Equations, 249 (2010), pp. 151-179.
	
	\bibitem{Feng}
	{\sc Y. Feng},
	Long time error analysis of the fourth-order compact finite difference methods for the nonlinear Klein-Gordon equation with weak nonlinearity, arXiv:2003.03951.
	
	%G
	
	\bibitem{Gau}
	{\sc  W. Gautschi},
	Numerical integration of ordinary differential equations based on trigonometric polynomials,
	Numer. Math., 3 (1961), pp. 381-397.
	
	\bibitem{Gr1}
	{\sc  V. Grimm},
	A note on the Gautschi-type method for oscillatory second-order differential equations,
	Numer. Math., 102 (2005), pp. 61-66.
	
	\bibitem{Gr2}
	{\sc  V. Grimm},
	On error bounds for the Gautschi-type exponential integrator applied to oscillatory second-order differential equations,
	Numer. Math., 100 (2005), pp. 71-89.
	
	%H
	
	\bibitem{HLu}
	{\sc  E. Hairer and Ch. Lubich},
	Spectral semi-discretisations of weakly nonlinear wave equations over long times,
	Found. Comput. Math., 8 (2008), pp. 319-334.	
	
	
     \bibitem{HGG}
	{\sc J. S. Hesthaven, S. Gottlieb and D. Gottlieb},
	Spectral Methods for Time-Dependent Problems,
	Cambridge University Press, Cambridge, New York, 2007.
	
	\bibitem{HL}
	{\sc  M. Hochbruck and Ch. Lubich},
	A Gautschi-type method for oscillatory second-order differential equations,
	Numer. Math., 83 (1999), pp. 402-426.
	
	%J
	 \bibitem{J}
	{\sc S. Jin},
	Efficient asymptotic-preserving (AP) schemes for some multiscale kinetic equations,
	SIAM J. Sci. Comput., 21 (1999), pp. 441-454.	
	
	%K
	
     \bibitem{KT}
	{\sc M. Keel and T. Tao},
	Small data blow-up for semilinear Klein-Gordon equations,
	Amer. J. Math., 121 (1999), pp. 629-669.	
	
	 \bibitem{K}
	{\sc S. Klainerman},
	Global existence of small amplitude solutions to nonlinear Klein-Gordon equations in four space-time dimensions,
	Comm. Pure Appl. Math., 38 (1985), pp. 631-641.	
	
	%L
	
	 \bibitem{Landau}
	{\sc H. J. Landau},
	Necessary density conditions for sampling and interpolation of certain entire functions,
	Acta Math., 117 (1967), pp. 37-52.	
	
     \bibitem{LH}
	{\sc H. Lindblad},
	On the lifespan of solutions of nonlinear Klein-Gordon equations with small initial data,
	Comm. Pure Appl. Math., 43 (1990), pp. 445-472.	
	
	%O
	\bibitem{O}
	{\sc K. Ono},
	Global existence and asymptotic behavior of small solutions for semilinear dissipative Klein-Gordon equations,
	Discrete Contin. Dyn. Syst., 9 (2003), pp. 651-662.	
	
	
	%S	
	\bibitem{SJJ}
	{\sc J. J. Sakurai},
	Advanced Quantum Mechanics,
	Addison-Wesley, New York, 1967.
	
	\bibitem{SB}
	{\sc  R. Sassaman and A. Biswas},
	Soliton perturbation theory for phi-four model and nonlinear Klein-Gordon equations,
	Commun. Nonlinear Sci. Numer. Simulat., 14 (2009), pp. 3239-3249.
	
	\bibitem{Shannon1}
	{\sc  C. E. Shannon},
	A mathematical theory of communication,
	Bell Syst. Tech. J., 27 (1948), pp. 379-423.
	
    \bibitem{Shannon2}
	{\sc  C. E. Shannon},
	A mathematical theory of communication,
	Bell Syst. Tech. J., 27 (1948), pp.623-666.
	
	
	\bibitem{ST}
	{\sc J. Shen and T. Tang},
	Spectral and High-Order Methods with Applications,
	Science Press, Beijing, 2006.
	
	\bibitem{SG}
	{\sc G. D. Smith},
	Numerical Solution of Partial Differential Equations: Finite Difference Methods,
	Oxford, UK, 1985.
	
	
	\bibitem{S}
	{\sc W. Strauss},
	Nonlinear wave equations,
	CBMS Reg. Conf. Ser. Math., 73, Amer. Math. Soc., Providence, RI, 1989.

	
	\bibitem{SV}
	{\sc  W. Strauss and L. V\'azquez},
	Numerical solution of a nonlinear Klein-Gordon equation,
	J. Comput. Phys., 28 (1978), pp. 271-278.
	
       %T
       	\bibitem{TY}
	{\sc G. Todorova and B. Yordanov},
	Critical exponent for a nonlinear Klein-Gordon equation with damping,
	J. Differential Equations, 174 (2001), pp. 464-489.	

	\bibitem{V}
	{\sc V. Thom{\'e}e},
	Galerkin Finite Element Methods for Parabolic Problems,
	Springer, Berlin, 1997.
	
	%V
	\bibitem{VH}
	{\sc W. T. Van Horssen},
	An asymptotic theory for a class of initial-boundary value problems for weakly nonlinear wave equations with an application to a model of the galloping oscillations of overhead transmission lines,
	SIAM J. Appl. Math., 48 (1988), pp. 1227-1243.	

	%X
	\bibitem{XG}
	{\sc  C. Xiong, M. R. Good, Y. Guo, X. Liu and K. Huang}
	Relativistic superfluidity and vorticity from the nonlinear Klein-Gordon equation,
	Phys. Rev. D, 90 (2014), pp. 125019.
	
	
         %Y
         \bibitem{YRS}
	{\sc  W. Yi, X. Ruan and C. Su},
	Optimal resolution methods for the Klein-Gordon-Dirac system in the nonrelativistic limit regime,
	 J. Sci. Comput., 79 (2019), 1907-1935.

\end{thebibliography}

\end{document}